\documentclass[
final
]{dmtcs-episciences}

\pdfoutput=1

\usepackage[utf8]{inputenc}
\usepackage{subfigure}

\usepackage{amsmath,amsthm,amssymb,amsfonts,listings}

\usepackage[round]{natbib}

\newcounter{proofcount}
\newtheorem{theorem}{Theorem}
\newtheorem{lemma}[theorem]{Lemma}
\newtheorem{observation}[theorem]{Observation}
\newtheorem{proposition}[theorem]{Proposition}

\newtheorem{prob}[theorem]{Problem}

\newtheorem{conj}[theorem]{Conjecture}

\newtheorem{corollary}[theorem]{Corollary}
\newtheorem{cor}[theorem]{Corollary}

\theoremstyle{definition}
\newtheorem{definition}[theorem]{Definition}
\newtheorem{problem}[theorem]{Problem}

\newtheorem{caseEE}{Case}

\usepackage{algorithm}
\usepackage{algorithmicx}
\usepackage[noend]{algpseudocode}

\author{Stephan Dominique Andres\affiliationmark{1} \and Edwin Lock\affiliationmark{2}\thanks{Corresponding author}}
\title[Characterising and recognising game-perfect graphs]{Characterising and recognising game-perfect graphs}
\affiliation{
  Faculty of Mathematics and Computer Science, FernUniversität in Hagen, Germany\\
  Department of Computer Science, University of Oxford, United Kingdom}
\received{2018-10-31}

\revised{2019-4-2}

\accepted{2019-4-22}

\begin{document}
\publicationdetails{21}{2019}{1}{6}{4935}

\maketitle
\begin{abstract}
Consider a vertex colouring game played on a simple graph with $k$ permissible colours. Two players, a \emph{maker} and a \emph{breaker}, take turns to colour an uncoloured vertex such that adjacent vertices receive different colours. The game ends once the graph is fully coloured, in which case the maker wins, or the graph can no longer be fully coloured, in which case the breaker wins. In the game $g_B$, the breaker makes the first move. Our main focus is on the class of \emph{$g_B$-perfect} graphs: graphs such that for every induced subgraph $H$, the game $g_B$ played on $H$ admits a winning strategy for the maker with only $\omega(H)$ colours, where $\omega(H)$ denotes the clique number of~$H$. Complementing analogous results for other variations of the game, we characterise $g_B$-perfect graphs in two ways, by forbidden induced subgraphs and by explicit structural descriptions. We also present a clique module decomposition, which may be of independent interest, that allows us to efficiently recognise $g_B$-perfect graphs.
\end{abstract}


\section{Introduction}
\subsection{The vertex colouring games}\label{section:thegames}
In a vertex colouring game first mentioned by \citet{gardner} and formally introduced by \citet{bodlaender}, two players take turns to colour an uncoloured vertex of a simple (undirected) graph $G$ with one of $k$ permissible colours such 
that adjacent vertices receive different colours. One player, Alice (the maker), aims to achieve a complete graph colouring, while the other player, Bob (the breaker), attempts to prevent this from happening by ensuring that some uncoloured 
vertex has neighbours coloured in all $k$ colours. If Alice succeeds in finding a strategy that forces Bob to cooperate in colouring the whole graph, she wins, otherwise Bob wins.

We denote the games in which Alice and Bob start by $g_A$ and $g_B$, respectively. It also turns out to be useful to consider games where either Alice or Bob is permitted to miss their turn, leading to four new games $g_{A,A}$, $g_{A,B}$, $g_{B,A}$ and $g_{B,B}$. Here the first entry of the index denotes the starting player and the second entry indicates the player who may miss any number of turns; in particular, they may also miss their first turn.
For any game~$g$ of the six games defined above, the \emph{$g$-chromatic number} $\chi_g(G)$ of a graph~$G$ denotes the minimum number of colours required for Alice to win the game on~$G$. In this paper, all graphs are simple and undirected.


\subsection{Motivation} 
Graph colouring games have received a great deal of attention over the last three decades~\citep{dunn2017,tuzazhu}. One area of interest has been to identify good upper bounds for the \mbox{$g_A$-chro}\-ma\-tic number 
(also known as \emph{game-chromatic number}) of certain classes of 
graphs. \citet{faigle1} showed that $\chi_{g_A}(F) \leq 4$ for any forest $F$, and more recently \citet{zhurefined} proved that \mbox{$\chi_{g_A}(G) \leq 17$} if $G$ is planar. Other graph classes with known constant upper bounds include 
cactuses~\citep{cactuses}, partial $k$-trees~\citep{zhupseudopartial} and outerplanar graphs~\citep{guanzhu}. These bounds are known to be tight only for forests and cactuses.

For other graph classes, upper bounds for $\chi_{g_A}$ are known only as a function of the clique number~$\omega(\cdot)$. The first such result was obtained by \citet{faigle1}, who proved that $\chi_{g_A}(I)\le3\omega(I)-2$ for any interval 
graph~$I$. Subsequently, upper bounds in terms of the clique number were also found for line graphs of various $k$-degenerate graph classes \citep{caizhu, erdoesetal}, Husimi trees \citep{sidorowiczhusimi}, and 
various incidence graphs~\citep{charpentiersopena}.

Many of the upper bounds above are the result of studying the \emph{colouring number}, a game invariant associated with the `colourblind' \emph{marking game} introduced by \citet{zhuplanar}, and exploiting the fact that the colouring number is 
an upper bound for the game-chromatic number for any graph. In order to tighten specific bounds for the game-chromatic number it may be necessary to design winning strategies for Alice that are not `colourblind'. A further `first-fit' variant of the graph colouring game is the Grundy colouring game introduced by \citet{havetzhu}. The game-chromatic number has also been studied in the context of random graphs \citep{bohman}.

%
While much of the literature on the vertex colouring game has focussed on the game $g_A$, there can be large discrepancies between the $g$-chromatic numbers of the different game variants $g$. Indeed, some effects of allowing a player to skip 
moves have been analysed by~\citet{zhucartesian} for the marking game.

In his original paper on vertex colouring games, \citet{bodlaender} asked about the complexity of deciding whether Alice can win the game $g_A$ on a graph $G$ with $k$ colours. This problem is in P for $k \leq 2$ (cf.~\citet[Theorems 3, 15, 17, 18]{andres4}). While it is easy to see that this decision problem is in PSPACE by constructing an alternating algorithm that simulates the game, the question of PSPACE-hardness remains open for all $k \geq 3$ and all the game variants mentioned above. In light of the lack of progress on the complexity of the colouring games in the general case, one might seek to restrict oneself to graph classes in which the games can be decided efficiently. This approach mirrors results achieved in the classic, non-competitive graph colouring setting.

It is well-known that deciding whether a given graph admits a proper colouring with $k\ge3$ colours is NP-hard \citep{karp}. For this reason, restricted classes of graphs that can be coloured efficiently are of major interest \citep{golumbic}. Perhaps the most well-known such class, the perfect graphs, has been the subject of several seminal results. A graph $G$ is considered to be \emph{perfect} if $\omega(H) = \chi(H)$ for all induced subgraphs $H$ of $G$, where $\chi(\cdot)$ denotes the chromatic number and $\omega(\cdot)$ denotes the clique number. \citet{gls} proved that colouring perfect graphs is in P. More recently, \citet{chudnovsky2} showed that recognising perfect graphs can be achieved in polynomial time. One year later, the famous Strong Perfect Graph Theorem by \citet{strongperfectgraphtheorem} characterised perfect graphs by means of forbidden induced subgraphs.


\citet{andres4} introduced the notion of game-perfect graphs with respect to any of the six games defined in Section~\ref{section:thegames}. For any such game $g$, a graph $G$ is \emph{game-perfect} with regard to $g$ (or simply \emph{$g$-perfect}) if 
\[
\chi_g(H) = \omega(H)
\]
for all induced subgraphs $H$ of $G$. It is easy to see that $\omega(G) \leq \chi(G) \leq \chi_g(G)$ for any graph $G$ and game~$g$, which implies that the game-perfect graphs are a subset of the perfect graphs.

In analogy to the Strong Perfect Graph Theorem, \citet{andres1} obtained the following characterisations for $g_A$, $g_{A,B}$ and $g_{B,B}$-perfect graphs. $P_4$ and $C_4$ denote the path and the cycle on 4 vertices, respectively. The other forbidden graphs in question are depicted in Figures~\ref{fig:gB-forbidden-graphs} and~\ref{fig:gA-forbidden-graphs}, while the graph class $E_1$ is defined in Section~\ref{section:graph-classes}.

\begin{theorem}[\citet{andres1}]\label{thm:gBB-perfect-characterisation}
	For any graph $G$, the following are equivalent.
	\begin{enumerate}[(i)]
		\item $G$ is $g_{B,B}$-perfect.
		\item $G$ contains no induced $P_4, C_4$, split 3-star or double fan (see Figure~\ref{fig:gB-forbidden-graphs}).
		\item Every connected component $C$ of $G$ is an instance of the graph class $E_1$ (see Figure~\ref{fig:explicit-structures}).
	\end{enumerate}
\end{theorem}

\begin{theorem}[\citet{andres1}]\label{thm:gA-gAB-perfect-characterisation}
	For any graph $G$, the following are equivalent.
	\begin{enumerate}[(i)]
		\item $G$ is $g_A$-perfect.
		\item $G$ is $g_{A,B}$-perfect.
		\item $G$ contains none of the following as induced subgraphs: a $P_4, C_4$, triangle star, $\Xi$-graph, the union of two double fans, the union of two split 3-stars or the union of a double fan with a split 3-star 
(see Figure
~\ref{fig:gA-forbidden-graphs}).
		\item If $C_1, \ldots, C_k$ are the connected components of $G$ and $k\ge1$, then without loss of generality $C_1$ contains a dominating vertex~$v$ such that
$G-v$ is
$g_{B,B}$-perfect.
	\end{enumerate}
\end{theorem}

Furthermore, the following holds for \emph{disconnected} graphs.
\begin{theorem}[\citet{andres1}]\label{thm:disconnected-gB-perfect-characterisation}
	Disconnected $g_B$-perfect graphs are $g_{B,B}$-perfect.
\end{theorem}

We note that, by definition, the classes of $g$-perfect graphs are hereditary, whereas in general, graphs may have a smaller $g$-chromatic number than some of their induced subgraphs.
It is also worth highlighting that
the $g_A$, $g_{A,B}$ and 
$g_{B,B}$-perfect graphs are all {trivially perfect} (cf.~\citet{golumbic,golumbictrivially,wolk}), whereas the class of $g_B$-perfect graphs is not, suggesting a richer family of structures. 


\subsection{Our results}
In this paper, we provide two characterisations of $g_B$-perfect graphs, one in terms of forbidden induced subgraphs and one by means of explicit structural descriptions. This constitutes the main result of this paper, extending Theorem~\ref{thm:disconnected-gB-perfect-characterisation} and complementing Theorems~\ref{thm:gBB-perfect-characterisation} and~\ref{thm:gA-gAB-perfect-characterisation}, which provide a characterisation of $g_A$, $g_{A,B}$ and $g_{B,B}$-perfect graphs.

\begin{figure}[htb]
    \subfigure[$F_1$: chair\label{fig:chair}]{\includegraphics[scale=1.0]{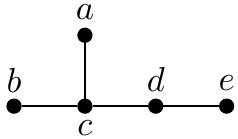}}\hfill
    \subfigure[$F_2$: path $P_5$\label{fig:5-path}]{\includegraphics{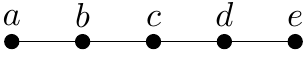}}\hfill
    \subfigure[$F_3$: $P_4 \cup K_1$\label{fig:4-path-K1}]{\includegraphics{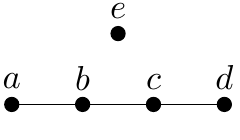}}\hfill
    \subfigure[$F_4:$ 4-fan\label{fig:4-fan}]{\includegraphics{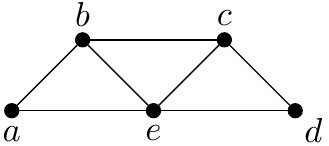}}
\par\medskip
    \subfigure[$F_{5}:$ split 3-star\label{fig:split-3-star}]{\includegraphics{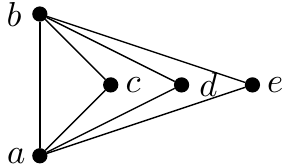}}\hfill
    \subfigure[$F_{6}:$ cycle $C_5$\label{fig:5-cycle}]{\includegraphics{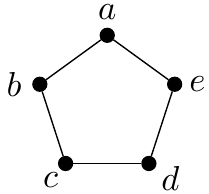}}\hfill
    \subfigure[$F_{7}:$ $C_4\cup K_1$\label{fig:4-cycle-K1}]{\includegraphics{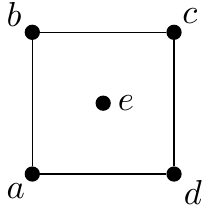}}\hfill
    \subfigure[$F_{8}:$ 4-wheel\label{fig:4-wheel}]{\includegraphics{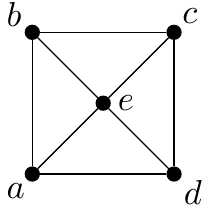}}
\par\medskip
    \subfigure[$F_{9}:$ double fan\label{fig:double-fan}]{\includegraphics{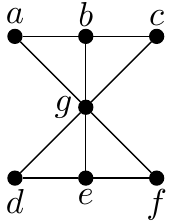}\hspace*{0.5cm}}\hfill
    \subfigure[$F_{10}$\label{fig:F10}]{\includegraphics{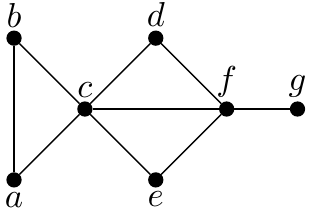}}\hfill
    \subfigure[$F_{11}$\label{fig:F11}]{\includegraphics{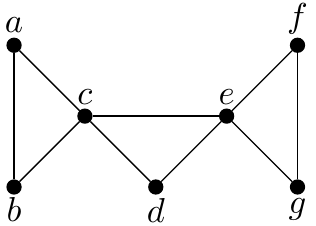}}\hfill
    \subfigure[$F_{12}$\label{fig:F12}]{\includegraphics{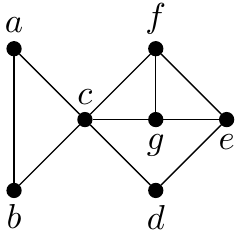}}\hfill
\par\medskip
    \subfigure[$F_{13}$\label{fig:F13}]{\includegraphics{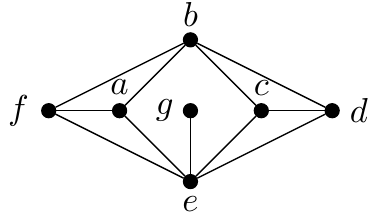}}\hfill
    \subfigure[$F_{14}$\label{fig:F14}]{\includegraphics{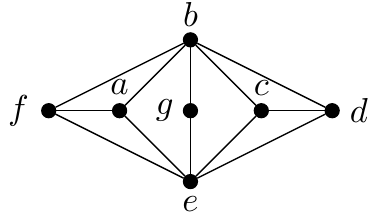}}\hfill
    \subfigure[$F_{15}$\label{fig:F15}]{\includegraphics{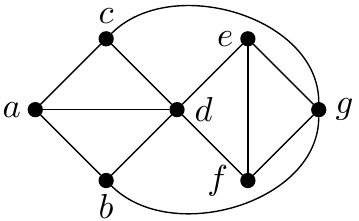}}
\caption{\label{fig:gB-forbidden-graphs}The fifteen forbidden induced subgraphs for $g_B$-perfect graphs.}
\end{figure}

\begin{theorem}[main result]\label{thm:gB-perfect-characterisation}
	For any graph $G$, the following are equivalent.
	\begin{enumerate}[(i)]
		\item $G$ is $g_B$-perfect.
		\item $G$ contains no induced $F_1, \ldots, F_{15}$ (see Figure~\ref{fig:gB-forbidden-graphs}).
		\item $G$ is an instance of one of the graph classes $E_1^{\cup}, E_2, \ldots, E_9$ (see Section~\ref{section:graph-classes} and Figure~\ref{fig:explicit-structures}).
	\end{enumerate}
\end{theorem}

\begin{figure}[htb] \subfigure[$\Xi$-graph\label{fig:xi-graph}]{\includegraphics{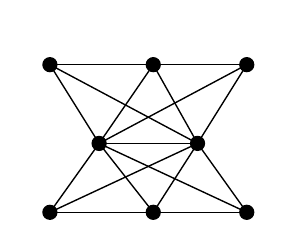}}\hfill
    \subfigure[triangle star\label{fig:triangle-star}]{\includegraphics{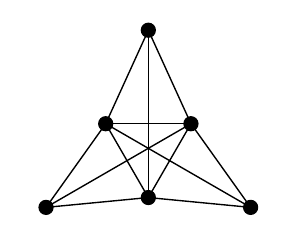}}\hfill
    \subfigure[two split 3-stars\label{fig:two-split-3stars}]{\parbox[b]{2.5cm}{\includegraphics{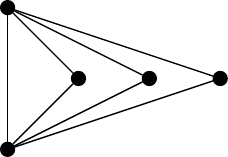}\\\includegraphics{figs/unlabelled-split-3-star.pdf}}}\hfill
    \subfigure[two double fans\label{fig:two-double-fans}]{\hspace*{0.3cm}\parbox[b]{1.7cm}{\includegraphics{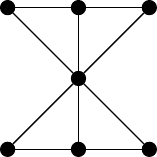}\par\medskip\includegraphics{figs/unlabelled-double-fan.pdf}}\hspace*{0.3cm}}\hfill
    \subfigure[split 3-star + double fan\label{fig:split-3-star-double-fan}]{\hspace*{0.7cm}\parbox[b]{2.3cm}{\includegraphics{figs/unlabelled-split-3-star.pdf}\par\medskip\includegraphics{figs/unlabelled-double-fan.pdf}}\hspace*{0.7cm}}
\caption{Some forbidden induced subgraphs for $g_A$-perfect graphs.}
\label{fig:gA-forbidden-graphs}
\end{figure}

We give a brief overview of the proof of Theorem~\ref{thm:gB-perfect-characterisation} and refer to the following sections for the remaining parts of the proof.

\begin{proof}[overview] We can assume that $G$ has two or more vertices. The implication (i) $\Rightarrow$ (ii) is proved in Theorem~\ref{thm:forbidden-graphs} of Section~\ref{section:forbidden}, where we give winning strategies for Bob on $F_1, \ldots, F_{15}$ with $\omega(F_i)$ colours.

For the implication (ii) $\Rightarrow$ (iii), let $G$ be a graph without induced $F_1, \ldots, F_{15}$. If $G$ is disconnected, it contains no induced $P_4$ or $C_4$,
hence, by Theorem~\ref{thm:gBB-perfect-characterisation}, each component is an
instance of $E_1$, implying $G \in E_1^\cup$.
Now assume that $G$ is connected. In Section~\ref{section:edge-decomposition}, we show that $G$ has a dominating edge and perform a dominating edge decomposition as well as a structural analysis of $G$. This leads to the case distinctions of 
Lemma~\ref{lemma:caseanalysis} in Section~\ref{section:case-distinctions}, classifying $G$ as an instance of $E_1^\cup,E_2, \ldots, E_9$.

Finally, to prove (iii) $\Rightarrow$ (i), let $G$ be an instance of~$E_1^\cup,E_2, \ldots, E_9$ and let $H$ be an induced subgraph of~$G$. By Lemma~\ref{lemma:E-is-hereditary} of Section~\ref{section:graph-classes}, $H$ is also an instance of~$E_1^\cup, E_2, \ldots, E_9$. In Section~\ref{section:strategies}, we present a strategy for each of these graph classes that allows Alice to win on any graph $H$ in that class with $\omega(H)$ colours, implying that $G$ is $g_B$-perfect.
\end{proof}

Recall that the $g_B$-chromatic number is lower-bounded by the clique number.
Our main result, Theorem~\ref{thm:gB-perfect-characterisation}, identifies a class of graphs for which the two invariants coincide, thus establishing a tight upper bound.
Whereas most results in the literature rely on the marking game to establish similar upper bounds,
Theorem~\ref{thm:gB-perfect-characterisation} is achieved by providing strategies for Alice that are not `colourblind' but rely on Alice's ability to recognise the specific colours that Bob has used.

Both characterisations obtained in Theorem \ref{thm:gB-perfect-characterisation} are instructive from an algorithmic perspective, as they each facilitate polynomial time checking of $g_B$-perfectness. The forbidden subgraph characterisation 
immediately yields the following $\Theta(n^7)$-time algorithm, where $n$ is the order of the graph: given a graph~$G$, check for any subgraph of $G$ with 5 or 7 vertices whether it is one of the fifteen forbidden graphs. Improving on this, we 
introduce a clique module decomposition technique in Section~\ref{section:complexity-results} which, together with the explicit structural characterisation of 
Theorems~\ref{thm:gBB-perfect-characterisation}-\ref{thm:gB-perfect-characterisation}, allows us to formulate an $O(n^2)$ time algorithm for checking whether a graph is $g_A$- or $g_B$-perfect. This yields the following complexity 
results, which are proved in Section~\ref{subsec:complexity-results}.

\begin{theorem}\label{thm:recognition}
	There is an $O(n^2)$ time algorithm deciding whether a graph $G$ with $n$ vertices is $g_B$-perfect (or $g_A$-perfect).
\end{theorem}

\begin{corollary}\label{corollary:alice-winning-time}
Alice can win on any $g_A$- or $g_B$-perfect graph $G$ with $\omega(G)$ colours using only $O(n^2)$ computational time.
\end{corollary}
\enlargethispage{0.5cm}


It is a standard exercise to show that \textsc{Hamilton Cycle}, the problem of deciding whether a graph has a Hamilton cycle, is NP-complete even for bipartite graphs \citep{krishnamoorthy}, 
which form a subset of the perfect graphs.
In Corollary~\ref{corollary:hamilton}, we see that this is no longer the case for game-perfect graphs. We would like to note that this complements results by \citet{hochstaettler1995} 
and \citet{babeletal}
concerning graphs with few $P_4$s.
Similarly, we expect other problems that are NP-complete for perfect graphs to be in P for game-perfect graphs.

\begin{corollary}\label{corollary:hamilton}
	\textsc{Hamilton Cycle} is in P for $g_A$- and $g_B$-perfect graphs.
\end{corollary}

\subsection{Organisation}
The rest of this paper is structured as follows. In Section~\ref{section:graph-classes} we introduce the classes occurring in the structural characterisation of $g_B$-perfect graphs in Theorem~\ref{thm:gB-perfect-characterisation}~(iii). Sections~\ref{section:forbidden} -- \ref{section:strategies} are devoted to the proof of Theorem~\ref{thm:gB-perfect-characterisation}. In Section~\ref{section:complexity-results}, we state our complexity results (Theorem~\ref{thm:recognition} and Corollary~\ref{corollary:hamilton}). We conclude by discussing the implications of our work to open problems in Section~\ref{section:further-work}.

\section{Notation and the classes $E_1$ to $E_9$}\label{section:graph-classes}
$P_n$ and $C_n$ denote the path and cycle graph with $n$ vertices, $K_n$ is the complete graph (or {clique}) with $n$ vertices and $K_{m,n}$ is the complete bipartite graph with vertex partitions of size $m$ and $n$. 
A graph is \emph{null} if it has no vertices. For any graph~$G$, let $V(G)$~and $E(G)$ denote the set of its vertices and edges, respectively. Denote by $N_G(v)$ the set of neighbours of $v$ in $G$ and by $N_G[v] := N_G(v) \cup \{v\}$ the set of neighbours of $v$ together with $v$ itself. We omit the subscript when $G$ is clear from context.

\begin{figure}[htb]
    \subfigure[$E_1$: ear animal\label{fig:E1}]{\includegraphics{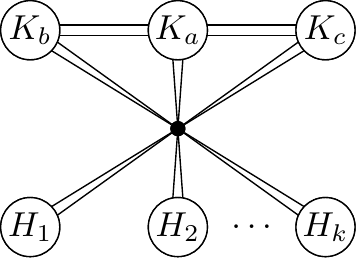}}\hfill
    \subfigure[$E_2$: ear bull\label{fig:E2}]{\includegraphics{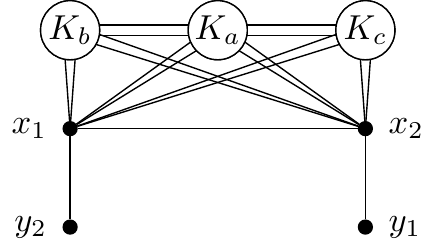}}\hfill
    \subfigure[$E_3$: expanded dragon / 4-wheel\label{fig:E3}]{\includegraphics{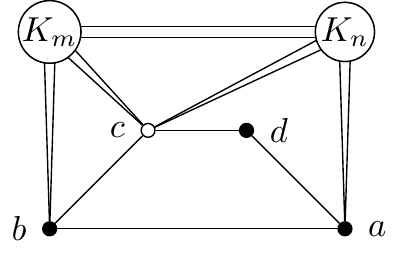}}\par\medskip
    \subfigure[$E_4$: expanded dragon / $K_{3,3}-e$\label{fig:E4}]{\hspace*{0.1cm}\includegraphics{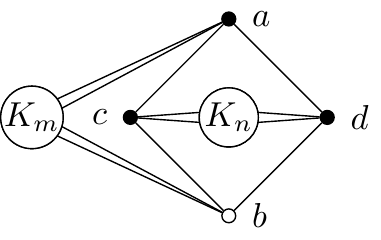}\hspace*{0.1cm}}\hfill
    \subfigure[$E_5$: expanded cocobi\label{fig:E5}]{\includegraphics{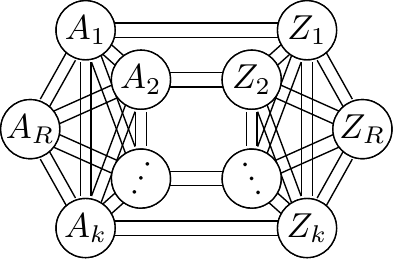}}\hfill
    \subfigure[$E_6$: expanded bull / house\label{fig:E6}]{\includegraphics{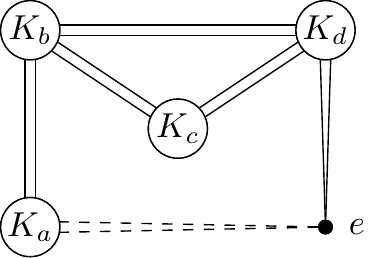}}\par\medskip
    \subfigure[$E_7$: exp. $K_{2,m}-e$ / $K_{2,m}$\label{fig:E7}]{\includegraphics{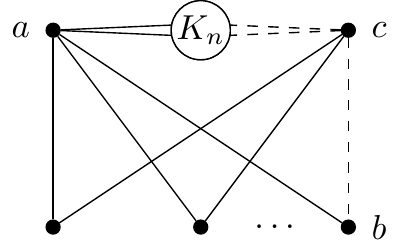}}\hfill
    \subfigure[$E_8$: expanded $K_{3,m}$\label{fig:E8}]{\includegraphics{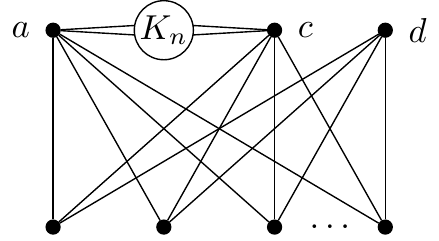}}\hfill
    \subfigure[$E_9$: (almost) complete bipartite\label{fig:E9}]{\hspace*{0.4cm}\includegraphics{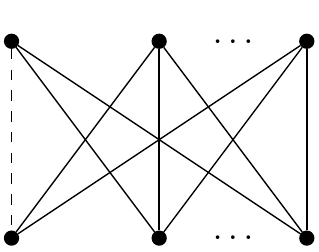}\hspace*{0.4cm}}
		\centering
		\includegraphics{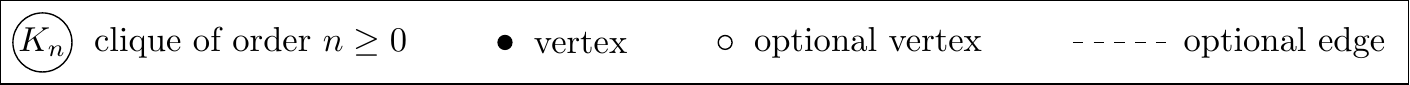}
	\caption{The nine structural possibilities for connected $g_B$-perfect graphs.}
	\label{fig:explicit-structures}
\end{figure}

\begin{table}[htb]
\begin{center}
\begin{tabular}{lll}
\begin{tabular}[t]{|l|l|}
\hline
\multicolumn{2}{|c|}{Class $E_1^\cup$}\\
\hline
$K$\phantom{cmcmc}&class of $G-K$\\
\hline
centre&$E_1^\cup$\\  
$K_a$&$E_1^\cup$\\  
$K_b$/$K_c$&$E_1^\cup$\\  
$H_i$&$E_1^\cup$\\
\hline
\end{tabular}\quad
&\begin{tabular}[t]{|l|l|}
\hline
\multicolumn{2}{|c|}{Class $E_{2}$}\\
\hline
$K$\phantom{cmcmc}&class of $G-K$\\
\hline
$x_1$/$x_2$&$E_1^\cup$\\
$y_1$/$y_2$&$E_1$\\  
$K_a$&$E_2$ by def.\\  
$K_b$/$K_c$&$E_6$\\  
\hline
\end{tabular}\quad
&\begin{tabular}[t]{|l|l|}
\hline
\multicolumn{2}{|c|}{Class $E_{3}$}\\
\hline
$K$\phantom{cmcmc}&class of $G-K$\\
\hline
$a$&$E_1^\cup$\\
$b$&$E_6$\\  
$c$&$E_3$ by def.\\
$d$&$E_5$\\
$K_m$&$E_7$ or $E_1$\\  
$K_n$&$E_6$ or $E_5$\\  
\hline
\end{tabular}\\[4ex]
\begin{tabular}[t]{|l|l|}
\hline
\multicolumn{2}{|c|}{Class $E_{4}$}\\
\hline
$K$\phantom{cmcmc}&class of $G-K$\\
\hline
$a$&$E_4$ or $E_1^\cup$\\
$b$&$E_4$ by def.\\  
$c$/$d$&$E_4$ or $E_5$\\
$K_m$&$E_7$\\  
$K_n$&$E_7$ or $E_1$\\  
\hline
\end{tabular}\quad
&\begin{tabular}[t]{|l|l|}
\hline
\multicolumn{2}{|c|}{Class $E_{5}$}\\
\hline
$K$\phantom{cmcmc}&class of $G-K$\\
\hline
$A_i$/$Z_i$&$E_5$ or $E_1^\cup$\\ 
$A_R$/$Z_R$&$E_5$ by def.\\ 
\hline
\end{tabular}\quad
&\begin{tabular}[t]{|l|l|}
\hline
\multicolumn{2}{|c|}{Class $E_{6}$}\\
\hline
$K$\phantom{cmcmc}&class of $G-K$\\
\hline
$e$/$K_a$&$E_1$\\
$K_b$/$K_d$&$E_5$ or $E_1^\cup$\\  
$K_c$&$E_5$\\  
\hline
\end{tabular}\\[4ex]
\begin{tabular}[t]{|l|l|}
\hline
\multicolumn{2}{|c|}{Class $E_{7}$}\\
\hline
$K$\phantom{cmcmc}&class of $G-K$\\
\hline
$a$&$E_1^\cup$\\
$b$&$E_7$ or $E_1^\cup$\\  
$c$&$E_1$\\
lower$\,\neq b$&$E_7$ by def.\\
$K_n$&$E_9$\\  
\hline
\end{tabular}\quad
&\begin{tabular}[t]{|l|l|}
\hline
\multicolumn{2}{|c|}{Class $E_{8}$}\\
\hline
$K$\phantom{cmcmc}&class of $G-K$\\
\hline
$a$/$c$&$E_7$\\
$d$&$E_7$\\
lower&$E_8$ by def.\\
$K_n$&$E_9$\\  
\hline
\end{tabular}\quad
&\begin{tabular}[t]{|l|l|}
\hline
\multicolumn{2}{|c|}{Class $E_{9}$}\\
\hline
$K$\phantom{cmcmc}&class of $G-K$\\
\hline
any&$E_9$ by def.\\
\hline
\end{tabular}
\end{tabular}
\end{center}
\caption{\label{table:hereditary} An illustration of the hereditary property of $E$. We list for each $G \in E_1^\cup,E_2, \ldots, E_9$ the class to which $G-K$ belongs, where $K$ is a maximal clique module of $G$. Two cases are 
given in some entries depending on the existence of the optional vertex or the optional edges.}
\end{table}

Given two graphs $G$ and $H$ with disjoint vertex sets, their \emph{union} $G \cup H$ is a graph on $V(G) \cup V(H)$ with edges $E(G) \cup E(H)$ and their \emph{join} $G \vee H$ is defined as the same graph with additional edges $vw$ for every pair $(v,w)\in V(G)\times V(H)$. If $G$ and $H$ are not disjoint, we implicitly make isomorphic vertex-disjoint copies and proceed as above. Two subsets $S, T \subseteq V(G)$ in a graph $G$ are \emph{completely connected} if each vertex in $S$ is adjacent to each vertex in $T$ and \emph{disconnected} if no edge exists between $S$ and $T$. (Hence $V(G)$ and $V(H)$ are completely connected in $G \vee H$ and disconnected in $G \cup H$.)
The subsets $S,T$ are \emph{partially connected} if they are neither completely connected nor disconnected.

We define the graph classes $E_1, \ldots, E_9$ by means of Figure~\ref{fig:explicit-structures} together with the following clarifications. Unless stated otherwise, large circles denote complete subgraphs of order at least 1 and 
small circles denote a single vertex. If the small circle is filled, then the vertex must be present in every graph of the graph class, otherwise (if it is hollow) it may be omitted. If two complete subgraphs or vertices are visually 
linked, they are completely connected, otherwise they are disconnected. Dashed lines indicate that the subgraphs or vertices may be either completely connected or disconnected. For $E_1$ we have $k \geq 0$ and for $E_5$ we 
have $k \geq 1$.

In the graph class $E_1$, we allow the large circles to be null graphs. Hence in terms of our notation introduced above, a graph $G$ is in $E_1$ if and only if it consists of a subgraph $H_0 := K_a \vee (K_b \cup K_c)$ with $a,b,c \geq 0$ and any number $k \geq 0$ of complete subgraphs $H_1, \ldots, H_k$ that are all completely connected to a dominating vertex. Additionally, we define $E_1^\cup$ as the graph class consisting of all graphs whose connected components are in $E_1$. 
In particular, the \emph{null graph} $K_0$ is a member of~$E_1^\cup$ (and of $E_9$). In $E_2$ we allow $K_a$ to be null, and in $E_5$ we allow $A_R$ and $Z_R$ to be null. Finally, in any graph $G \in E_7$, vertex~$c$ is completely connected to $K_n$ or vertex $b$ (or both).


Many of the $E_i$ can be considered as expanded forms of simple \emph{base graphs} obtained by replacing single vertices with complete graphs and respecting edge relations. In Subsection~\ref{subsec:clique-module-decomposition} we show that these base graphs can be recovered from an expanded graph by means of a clique module decomposition technique. The \emph{base graphs} of $E_5$, for instance, can be described as complements of complete bipartite graphs minus an almost maximal matching, or \emph{cocobi} for short. This gives rise to our names for classes $E_3$ to $E_9$. Finally, the names of $E_1$ and $E_2$ have historic reasons. In the following lemma, we show that the union of all our classes,
\[
E := E_1^\cup \cup \bigcup_{i=2}^9 E_i,
\]
is hereditary.

\begin{lemma}\label{lemma:E-is-hereditary}
	Any induced subgraph of a graph in $E$ is also in $E$.
\end{lemma}


\begin{proof}
    With Table~\ref{table:hereditary}, the reader will easily verify that removing a single vertex~$v$ from a graph $G$ in $E$ indeed results in a graph $G-v$ that is in $E$. Fix an induced subgraph $H$ of $G \in E$. Then removing 
vertices $V(G)\setminus V(H)$ from~$G$ one by one yields $H$ and the graph remains in $E$ after each step.
\end{proof}

\section{The forbidden induced subgraphs $F_1$ to $F_{15}$}\label{section:forbidden}
Here we prove the implication (i) $\Rightarrow$ (ii) of Theorem~\ref{thm:gB-perfect-characterisation}, restated as Theorem~\ref{thm:forbidden-graphs} below.

\begin{theorem}\label{thm:forbidden-graphs}
    If $G$ is a $g_B$-perfect graph, then it contains no induced $F_1, \ldots, F_{15}$ from Figure~\ref{fig:gB-forbidden-graphs}.
\end{theorem}
\begin{proof}
   It suffices to show for every $1 \leq i \leq 15$ how Bob can win the game $g_B$ on $F_i$ with $w(F_i)$ colours. As this is already known for $F_1, \ldots, F_9$ \citep{andres1}, we provide strategies for Bob on graphs $F_{10}, \ldots, F_{15}$ that allow him to win. All six graphs have clique number $\omega(F_i) = 3$. Without loss of generality, we refer to the vertex labels assigned in Figure~\ref{fig:gB-forbidden-graphs} and assume that the colours used appear in the order red, green and blue. A vertex is considered to be \emph{surrounded} if it has neighbours coloured in all $w(F_i)$ possible colours. Note that once any vertex is surrounded, Bob wins. The following observation gives a second condition that guarantees a win for Bob.
   
\begin{observation}\label{obs:uncoloured-P4-C4}
	Bob wins as soon as the remaining uncoloured vertices in the graph induce a $C_4$ or $P_4$ and admit colours green and blue but not red, and Alice is about to make the next move.
\end{observation}

\begin{description}
     \item[$F_{10}$:] Note that in any proper colouring with three colours, $d$ and~$e$ must be coloured the same. Hence Bob starts with~$e$ in red and Alice must respond with~$d$ in red. Bob then colours~$b$ in green. If Alice now colours~$a$ 
in red, then Bob can surround~$c$ by colouring~$f$ in blue. If Alice colours~$a$ in blue, then she herself surrounds~$c$. If she colours~$c$ in blue, then Bob wins by colouring~$g$ in green. If Alice colours $f$ or~$g$ in either possible 
colour, then Bob can colour~$a$ in blue and surround $c$.
    
     \item[$F_{11}$:] Bob begins with $d$ in red. Due to symmetry, there are three possibilities. If Alice colours $c$ green, then Bob responds with~$g$ in blue and surrounds~$e$. If Alice instead colours $a$ green, then Bob responds with~$e$ 
in blue and surrounds~$c$. If Alice colours $a$ red, then Bob colours~$g$ red and wins by Observation~\ref{obs:uncoloured-P4-C4} due to the $P_4$ induced by the remaining uncoloured vertices.

    \item[$F_{12}$:] Bob starts with $a$ in red. If Alice colours $c$ or $e$ green, then Bob colours the other vertex blue, surrounding $f$ and $g$. If Alice instead colours $b,d,f$ or $g$ green, then Bob colours another of these vertices 
blue, surrounding $c$. Finally, if she colours $f$ (or $g$) red, then Bob colours $d$ red and vice versa. By Observation~\ref{obs:uncoloured-P4-C4}, the $P_4$ induced by the remaining uncoloured vertices implies that Bob wins.
    
    \item[$F_{13}$ and $F_{14}$:] In any proper colouring with three colours, $b$ and~$e$ are coloured the same and~$g$ receives a different colour to~$b$ and~$e$. Bob begins with~$g$ in red. If Alice colours $e$ or $b$, then Bob colours 
the other vertex in a different colour and wins. Hence by symmetry we assume Alice colours $a$. If she colours $a$ green, then Bob colours $f$ blue and wins because $e$ is surrounded. If she colours $a$ red, then Bob responds with 
$c$ in red and wins by Observation~\ref{obs:uncoloured-P4-C4} due to the $C_4$ induced by the remaining uncoloured vertices.

	\item[$F_{15}$:] Observe that the pairs $b,c$ and $d,g$ must be uniformly coloured. Bob starts with $b$ in red. Alice colours~$c$ red to stop Bob from colouring~$c$ in a different colour on his next move. Bob then colours $f$ red and wins by Observation~\ref{obs:uncoloured-P4-C4} due to the $P_4$ induced by the remaining uncoloured vertices.
\end{description}
\end{proof}

\section{The dominating edge decomposition}\label{section:edge-decomposition}
In the following two sections we prove the implication (ii) $\Rightarrow$ (iii) of Theorem~\ref{thm:gB-perfect-characterisation}. We can assume that $G$ is connected and has order at least $2$. Indeed, for the graphs $K_0$ and $K_1$ the statement is trivially true. Furthermore, suppose $G$ is disconnected. Since $F_3$ and $F_7$ are forbidden and $G$ has at least two components, the graph~$G$ contains no induced $P_4$ or $C_4$. As $G$ also contains no split 3-star $(F_5)$ or double fan~$(F_9)$, Theorem~\ref{thm:gBB-perfect-characterisation} implies $G \in E_1^\cup$ and we are done. Note that this argument proves Theorem~\ref{thm:disconnected-gB-perfect-characterisation}.

Hence from now on, assume that $G$ is a connected graph of order at least $2$ that does not contain graphs $F_1$ to $F_{15}$ as induced subgraphs. In the remainder of this section we prove structural properties of $G$. These results form the technical basis of the proof of implication (ii) $\Rightarrow$ (iii) in Theorem~\ref{thm:gB-perfect-characterisation}, which is concluded by means of a series of structural case distinctions in Section~\ref{section:case-distinctions}.

First we prove that $G$ admits a decomposition into a dominating edge and three subgraphs $G_1, G_2$ and $G_3$, as shown in Figure~\ref{fig:edge-decomposition}. We then identify the structures of $G_1, G_2$ and $G_3$ (in 
Section~\ref{subsec:internal}) 
and the relationship 
between them (in Sections \ref{subsec:G1G2} and~\ref{subsec:G3}). This allows us to explicitly describe the structure of $G$ as belonging to one of the classes $E_1^\cup, E_2, \ldots, E_9$ in Section~\ref{section:case-distinctions}, 
Lemma~\ref{lemma:caseanalysis}. The following result by \citet{cozzens} provides the starting point for our decomposition.

\begin{figure}[htb]
	\centering
	\includegraphics{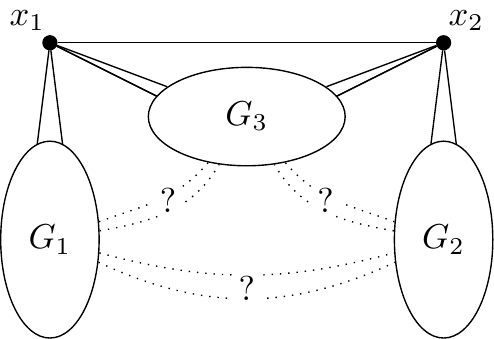}
	\caption{The dominating edge graph decomposition corresponding to Definition~\ref{def:G1-G2-G3}.}
	\label{fig:edge-decomposition}
\end{figure}

\begin{lemma}[\citet{cozzens}]\label{lemma:cozzens-kelleher}
    Connected graphs of order at least $2$ without an induced $P_5, C_5$ or $3$-spider with thin legs (see Figure \ref{fig:thin-spider}) have a dominating edge.
\end{lemma}

\begin{figure}[htb]
\hfill\subfigure[3-spider with thin legs\label{fig:thin-spider}]{\includegraphics{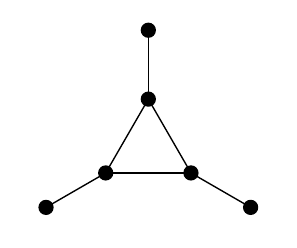}}
\hfill\subfigure[3-star\label{fig:3-star}]{\includegraphics{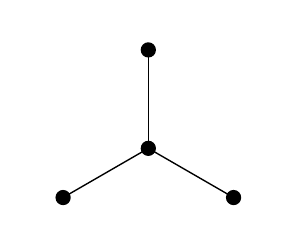}}\hfill
	\caption{(a) The 3-spider with thin legs is a forbidden graph for Lemma~\ref{lemma:cozzens-kelleher}.
(b) The 3-star is a forbidden graph for Lemma~\ref{lemma:noncomplete}.}
	\label{fig:cozzens-forbidden-graphs}
\end{figure}

\begin{corollary}\label{general:dominating_edge}
    Our graph $G$ has a dominating edge.
\end{corollary}
\begin{proof}
By assumption, $G$ has no induced $P_5$ or $C_5$. It also has no induced $3$-spider with thin legs, as this would imply an induced $F_3$ obtained by removing a vertex of degree~$3$ from the $3$-spider. Hence we can apply Lemma~\ref{lemma:cozzens-kelleher}.
\end{proof}

From now on, fix a dominating edge $x_1x_2$ in $G$. We define subgraphs $G_1, G_2$ and $G_3$ relative to this edge as follows.

\begin{definition}\label{def:G1-G2-G3}
$G_1$ is the subgraph of $G$ induced by the vertex set $N_G(x_1) \setminus N_G[x_2]$,
the graph $G_2$ is induced by the vertex set $N_G(x_2) \setminus N_G[x_1]$,
and $G_3$ is induced by the vertex set $N_G(x_1) \cap N_G(x_2)$.
\end{definition}

\subsection{The internal structures of $G_1, G_2$ and $G_3$}\label{subsec:internal}
The structures of $G_1, G_2$ and $G_3$ and the relationships between the three subgraphs satisfy certain rules that we express as a series of lemmas in the following three subsections. Note that by renaming $x_1$ to $x_2$, $G_1$ becomes $G_2$, and vice versa. In particular, all statements and proofs below concerning $G_1$ also hold for $G_2$ by symmetry. First of all, we describe the internal structures of $G_1, G_2$ and $G_3$. The following structural lemma turns out to be useful.

\begin{definition}
A graph $H$ is a \emph{possibly degenerate ear graph} if it has the structure $H = K_r \vee (K_p \cup K_q)$ with $p,q,r \geq 0$.
We call $H$ an \emph{ear graph} if $p,q,r \geq 1$ (see Figure~\ref{fig:simple-strategies}). 
\end{definition}

\begin{lemma}\label{lemma:noncomplete}
    A connected, non-complete graph $H$ of order at least~3 without an induced $P_4, C_4$ or 3-star (see Figure \ref{fig:3-star}) is an ear graph.
\end{lemma}

\begin{proof}[of Lemma~\ref{lemma:noncomplete}]
    We proceed by induction on the number of vertices. Let $H$ be a connected, non-complete graph. If $H$ has three vertices, then it is a $P_3$. Now suppose $H$ has order greater than three and pick a vertex $x$ in $H$ such that $H-x$ is 
connected. If $H-x$ is complete, then we are done.
Otherwise, we have $ H-x = K_r \vee (K_p \cup K_q)$ for some $p,q,r \geq 1$ by induction hypothesis. As $H$ is connected, $x$ is adjacent to some vertex in $K_p$ or $K_q$, or else $H$ contains an induced 3-star. Thus $x$ must be completely connected to $K_r$, or else $H$ contains a $P_4$ or $C_4$. Moreover, $x$ is either completely connected or disconnected from $K_p$ and $K_q$, or else $H$ contains a 3-star or $P_4$. Hence without loss of generality, 
$x$~is completely connected to $K_r$ and $K_p$ and either completely connected to or disconnected from $K_q$.
\end{proof}

Lemma~\ref{lemma:noncomplete} allows us to characterise the inner structures of $G_1, G_2$ and $G_3$ as follows.
\begin{lemma}\label{lemma:N-structure}
$G_1$ has at most one non-complete component $N$. Moreover, $N$ is an ear graph. 
\end{lemma}
\begin{proof}
If $G_1$ had two non-complete components, then each component would contain an induced $P_3$ which, together with $x_1$, would induce a double fan ($F_9$). Now let $N$ be a non-complete component of~$G_1$. We apply Lemma \ref{lemma:noncomplete}, as $N$ contains no induced $P_4, C_4$ or 3-star: a $P_4$ would induce an $F_3$ together with $x_2$, a $C_4$ would induce an $F_7$ together with $x_2$ and a 3-star would induce a split 3-star ($F_5$) 
together with~$x_1$.
\end{proof}

\begin{lemma}\label{lemma:G3-shape}
    $G_3$ is a possibly degenerate ear graph.
\end{lemma}
\begin{proof}
Note that no set of three vertices in $G_3$ is independent, or else it would induce a split 3-star together with $x_1$ and $x_2$. This implies that $G_3$ is either connected or has at most two components which are both complete. If $G_3$ is 
complete or has two components, then we are done. Hence let $G_3$ be connected and not complete. By assumption, $G_3$ has no induced 3-star, $P_4$ or~$C_4$, as they respectively induce a split 3-star ($F_5$),  4-fan ($F_4$) or 4-wheel ($F_8$) 
together with~$x_1$, and we can apply Lemma~\ref{lemma:noncomplete}.
\end{proof}

\subsection{The adjacency relations between $G_1$ and $G_2$}\label{subsec:G1G2}
In this subsection we study the possible adjacency relations between $G_1$ and $G_2$. Our first lemma establishes that $G_1$ and $G_2$ must be \emph{almost} completely connected in a specific, well-defined way. This result will be used frequently.

\begin{lemma}\label{lemma:all-but-one}
    Let $G_1$ and $G_2$ both be non-empty. With the exception of one pair of components $(X,Y)$, where~$X$ is a component of~$G_1$ and~$Y$ is a component of~$G_2$, every component of~$G_1$ is completely connected to every component of~$G_2$. $X$ and $Y$ may be completely connected, partially connected or disconnected.
\end{lemma}

Note that if $G_1$ and $G_2$ consist of the single components $X$ and $Y$, respectively, then Lemma~\ref{lemma:all-but-one}~trivially states that $G_1$ and $G_2$ may be disconnected, partially connected or completely connected.

\begin{proof}[of Lemma~\ref{lemma:all-but-one}]
        If both $G_1$ and $G_2$ consist of a single component, the statement holds vacuously. Hence assume without loss of generality that $G_1$ has at least two components $X$ and $X'$. Suppose neither $X$ nor $X'$ is completely connected 
to~$G_2$. Then there exists a vertex $a \in X$ not adjacent to some $v \in G_2$ and a vertex $b \in X'$ not adjacent to some $w \in G_2$. If $v=w$, then the vertices $a,b,w,x_1$ and $x_2$ induce a chair ($F_1$). Now assume that $v \not = w$; 
we distinguish between two cases due to symmetry. If there is no edge between~$a$ and~$w$, then the same vertices again induce a chair ($F_1$). If the edges $aw$ and $bv$ both exist, then the vertices $a,b,v,w$ and $x_1$ induce a $P_5$ ($F_2$) 
or $C_5$ $(F_6)$, depending on whether $v$ and $w$ are adjacent in $G_2$. Hence all but one component of $G_1$ must be completely connected to $G_2$ and by symmetry, all but one component in $G_2$ must be completely connected to $G_1$.
\end{proof}

The next lemma shows what constraints the existence of a non-complete component of $G_1$ places on~$G_2$.

\begin{lemma}\label{lemma:N-in-G1}
    If $G_1$ has an ear graph component $N=K_r\vee(K_p\cup K_q)$ with $p,q,r\ge1$, then the following holds.
\begin{itemize}
\item[(i)] $G_2$ has at most one vertex $v$.
\item[(ii)] If $G_2$ is a singleton graph, then it is completely connected to $K_p$ and $K_q$ and disconnected from~$K_r$, and $G_1$ has only the one component $N$.
\item[(iii)] If $p,q \geq 2$, then $G_2$ is null. 
\end{itemize}
\end{lemma}
\begin{proof}
       Let $a \in K_p, b \in K_r, c \in K_q$ be three vertices in $N$ that induce a $P_3$. We first show that any vertex~$v$ in $G_2$ must be completely connected to $K_p$ and $K_q$ and disconnected from $K_r$. If~$v$ is adjacent to none of the three vertices $a,b$ and $c$, then the vertices $a,c,v,x_1$ and $x_2$ induce a chair ($F_1$). If $v$ is adjacent to exactly one of the three, then the vertices $a,b,c,v$ and $x_2$ induce a chair $(F_1)$ or a $P_5$ ($F_2$). Finally, if~$v$ is adjacent to $a$ and $b$, to $b$ and $c$, or all three, then the vertices $a,b,c,v$ and $x_1$ induce a 4-fan ($F_4$) or 4-wheel ($F_8$). This proves the first part of (ii).
	
	Now suppose $G_2$ has two vertices $v$ and $w$. Then they are both adjacent to $a$ and $c$ and not adjacent to~$b$. If $v$ and $w$ are not adjacent, then the vertices $b,v,w,x_1$ and $x_2$ induce a chair~($F_1$), otherwise the vertices $a,c,v,w$ and $x_2$ induce a split 3-star~($F_5$). This proves (i).
	
	To prove the second part of (ii), let $v$ be the vertex of $G_2$, and suppose $d$ is a vertex in a second component of $G_1$. As $v$ and $N$ are not completely connected, Lemma~\ref{lemma:all-but-one} implies that $d$ is adjacent to $v$ and the vertices $a,b,d,v$ and $x_2$ induce a chair~($F_1$).

    Finally, to prove (iii), suppose that $p,q \geq 2$ and let $a'$ and $c'$ be additional vertices in $K_p$ and $K_q$, respectively. If $G_2$ contains a vertex $v$, then it is adjacent to $a,a',c,c'$ by (ii) and the vertices $a,a',c,c',v,x_1$ and $x_2$ induce an $F_{14}$.
\end{proof}

The next four lemmas are useful in situations where we know that neither $G_1$ nor $G_2$ is null.

\begin{lemma}\label{lemmaY}
If $G_1$ has two adjacent vertices $a$ and $b$, and $G_2$ has two non-adjacent vertices $u$ and $v$, then $u$ and $v$ cannot both be adjacent to $a$ and~$b$.
\end{lemma}

\begin{proof}
If the vertices $u$ and $v$ are both adjacent to $a$ and $b$,
then the vertices $a,b,u,v$ and $x_1$ induce a split $3$-star~($F_5$).
\end{proof}

\begin{lemma}\label{lemmaX}
    Let $G_1$ and $G_2$ respectively contain components $C$ and $W$ of order at least $2$. If $G_1$ has a vertex $d$ not in $C$, then $C$ and $W$ are either disconnected or completely connected.
\end{lemma}
\begin{proof}
Suppose $C$ and $D$ are partially connected. Then there exist adjacent vertices $a,b \in C$ and adjacent vertices $u,v \in W$ that are partially connected and by Lemma~\ref{lemma:all-but-one}, $d$ must be adjacent to $u$ and $v$. By Lemma~\ref{lemmaY}, neither vertex $a$ nor $b$ is adjacent to both $u$ and $v$, leading to two possibilities due to symmetry. If $a$ is adjacent to $u$ and $b$ is not, then the vertices $a,b,d,u$ and $x_2$ induce a chair~($F_1$). If both $a$ and $b$ are adjacent to $u$, then the vertices $a,b,d,u,v,x_1$ and $x_2$ induce an~$F_{15}$.
\end{proof}

\pagebreak[4]

\begin{lemma}\label{lemma:one-component-of-size-2}
    Unless $G_2$ is null, $G_1$ has at most one component of order at least $2$.
\end{lemma}
\begin{proof}
    Assume $G_1$ has two components of order at least $2$ and let~$v$ be a vertex of~$G_2$. Pick two adjacent vertices from each of the two components in $G_1$ and denote them by $a,b$ and $c,d$, respectively. Due to Lemma~\ref{lemma:all-but-one}, we assume without loss of generality that $v$ is adjacent to $a$ and $b$, which leads to three possibilities due to symmetry. If~$v$ is adjacent to neither $c$ nor~$d$, then the vertices $a,b,c,d,v,x_1$ and $x_2$ induce an~$F_{12}$. If $v$ is adjacent to $c$ and not to $d$, then the vertices $a,c,d,v$ and $x_2$ induce a chair ($F_1$). If $v$ is adjacent to both $c$ and $d$, then the vertices $a,b,c,d,v,x_1$ and $x_2$ induce an $F_{14}$.
\end{proof}

\begin{lemma}\label{lemma:at-most-two-components}
    If $G_1$ has a component $A$ of order at least $2$, then $G_2$ has at most two components.
\end{lemma}
\begin{proof}
    Let $A$ be a component of $G_1$ of order at least $2$ and $a,b$ be two adjacent vertices in $A$. Suppose $G_2$ has three components $C,D$ and $E$. Then without loss of generality, $A$ is completely connected to $C$ and $D$ by Lemma \ref{lemma:all-but-one}, so $a$ and $b$ together with any $c \in C$, $d \in D$ and $x_1$ induce a split 3-star~($F_5$).
\end{proof}

The following important structural result holds if both $G_1$ and $G_2$ are complete subgraphs.
\begin{lemma}\label{lemma:Kn-to-Kn}
    If $G_1$ and $G_2$ are both non-null complete subgraphs, then $G_1$ and~$G_2$ can be partitioned into subgraphs $G_1 = A_R, A_1, \ldots, A_k$ and $G_2 = Z_R, Z_1, \ldots, Z_k$ for some $k \geq 0$ such that
	\begin{itemize}
		\item for all $i,j \in \{1, \ldots, k \}$ with $i \not = j$, $A_i$ is completely connected to $Z_i$ and disconnected from $Z_j$.
 		\item $A_R$ and $Z_R$ are disconnected from all $Z_i$ and $A_i$, respectively, and from each other.
	\end{itemize}
This is illustrated by the graph class $E_5$ in Figure~\ref{fig:explicit-structures}.
\end{lemma}
\begin{proof}
    If both $G_1$ and $G_2$ consist of a single vertex, then the lemma is trivially true. Hence assume that $G_1$ has at least two vertices $a$ and~$b$, and let $S_a$ be the set of vertices in $G_2$ adjacent to $a$. Note that $b$ is either 
completely connected to or disconnected from $S_a$. To see this, suppose $v \in S_a$ is adjacent to $a$ and $b$, and $w \in S_a$ is adjacent only to $a$. As $G_1$ and $G_2$ are complete, the vertices $a,b,v,w$ and $x_1$ induce a 4-fan~($F_4$). 
This 
fact implies that for every two vertices $a,b \in G_1$, $S_a$ and $S_b$ are either disjoint or equal, which partitions $G_2$ as follows. The $Z_1, \ldots, Z_k$ are all the possible distinct subgraphs of $G_2$ induced by $S_a$ for some $a \in G_1$, while $Z_R$ is the subgraph of $G_2$ induced by all vertices not adjacent to $G_1$.
	
	From the partition of $G_2$ we also obtain a partition $A_R, A_1, \ldots, A_k$. Clearly, every vertex $a$ in $G_1$ is adjacent to at most one subgraph $Z_i$, $1 \leq i \leq k$. Hence we define by $A_i$ the subgraph induced by the subset of all vertices in $G_1$ adjacent to $Z_i$, and by $A_R$ the subgraph of $G_1$ induced by all vertices not adjacent to any vertex in $G_2$.
\end{proof}

\subsection{The relationship between $G_1/G_2$ and $G_3$}\label{subsec:G3}
Now that we have seen how different configurations of $G_1$ can be connected to~$G_2$ and vice versa, regardless of the shape of~$G_3$, we investigate how the existence and configuration of a non-null~$G_3$ constrains the other two subgraphs.  The following lemmas refer to a series of case distinctions on whether $G_3$ or $G_1/G_2$ contain a non-complete component. If $G_3$ is not complete, the case is simple and completely described by Lemma~\ref{lemma:G3-not-clique}.

\begin{lemma}\label{lemma:G3-not-clique}
    If~$G_3$ is not complete, $G_3$ is disconnected from $G_1$ and $G_2$. Moreover, neither~$G_1$ nor~$G_2$ contain a non-complete component.
\end{lemma}
\begin{proof}
     Let $y_1$ and $y_2$ be two non-adjacent vertices in $G_3$, and let $a$ be a vertex in $G_1$ (or $G_2$). If $a$ is adjacent to $y_1$ and/or $y_2$, then the vertices $a,x_1,x_2,y_1$ and $y_2$ induce a 4-fan ($F_4$) or 4-wheel~($F_8$). This 
shows that vertices $y \in G_3$ not dominating $G_3$ are disconnected from $G_1$ and $G_2$. Now suppose $G_3$ has a vertex $y_3$ dominating $G_3$. If $a$ is adjacent to~$y_3$, then the vertices $a,x_1,y_1,y_2$ and $y_3$ induce a split 
3-star~($F_5$). Hence no vertex in $G_3$ can be adjacent to $G_1$ or $G_2$. Secondly, suppose $G_1$ has a non-complete component $N$. As this $N$ contains an induced $P_3$ and none of the vertices in this $P_3$ are adjacent to $y_1$ or~$y_2$, 
the $P_3$ and the vertices $x_1, x_2, y_1$ and $y_2$ induce a double fan~($F_9$).
\end{proof}

If $G_3$ is a non-null complete subgraph, the situation is a little more complex. 
The analysis of this case comprises the remainder of this subsection.
We first assume that $G_1$ has a non-complete component. 
\begin{lemma}\label{lemma:G3-if-N-then-independent}
    Let $G_3$ be non-null and complete. If $G_1$ contains a non-complete component $N$, then $G_1$ and $G_3$ are disconnected and $G_2$ is null.
\end{lemma}
\begin{proof}
	Recall from Lemma~\ref{lemma:N-structure} that $N = K_r \vee (K_p \cup K_q)$ and let $a \in K_p, b \in K_r, c \in K_q$, as well as $y \in G_3$. Firstly, note that $N$ and $G_3$ are disconnected. To see this, suppose $y$ is adjacent to some non-empty subset of $\{a,b,c\}$. This yields two cases due to symmetry. If $y$ is adjacent to $a$ and not adjacent to~$b$ (or vice versa), then the vertices $a,b,x_1,x_2$ and $y$ induce a 4-fan ($F_4$). If $y$ is adjacent to $a,b$ and $c$, then the vertices $a,c,x_1,x_2$ and $y$ induce a split 3-star~($F_5$). Secondly, let $d \in G_1$ be a vertex not in $N$ and assume it is adjacent to $y \in G_3$. Then the vertices $a,b,c,d,x_1, x_2$ and $y$ induce a double fan~($F_9$). Lastly suppose that $G_2$ has a vertex $v$. By Lemma~\ref{lemma:N-in-G1}, $v$ is adjacent to $a,c$ and not adjacent to $b$. If $v$ is not adjacent to $y$, then the vertices $a,c,v,x_2$ and $y$ induce a chair~($F_1$) and if $v$ is adjacent to $y$, then the vertices $a,b,c,v,x_1,x_2$ and $y$ induce an $F_{15}$.
\end{proof}

From now on we assume that neither $G_1$ nor $G_2$ has a non-complete component. $G_3$ is still complete and non-null.

\begin{lemma}\label{lemma:G3-connected-to-one}
    Let $G_3$ be a non-null complete subgraph. 
\begin{itemize}
\item[(i)] At most one component of $G_1$ is connected to $G_3$. The same holds for $G_2$. 
\item[(ii)] If $A$ and $V$ are respective components of $G_1$ and $G_2$ that are both connected to $G_3$, then $A$ is connected to different vertices of $G_3$ than $V$.
\end{itemize}
\end{lemma}
\begin{proof}
    (i) Suppose $a$ and $b$ are vertices from two components of $G_1$ connected to $G_3$. If~$a$ and~$b$ are both adjacent to the same vertex~$y$ in~$G_3$, then the vertices $a,b,x_1, x_2$ and $y$ induce a split 3-star~($F_5$). If~$a$ is adjacent to $y_1 \in G_3$ and $b$ is adjacent to $y_2 \in G_3$, then the vertices $a,b,x_1, y_1$ and $y_2$ induce a 4-fan~($F_4$).
    
    (ii) Suppose $G_3$ is connected to a component in $G_1$ and a component in~$G_2$. Thus there exist two vertices $a \in G_1$ and $v \in G_2$ that are both adjacent to a vertex in $G_3$. If $a$ and $v$ are both adjacent to the same vertex $y$ in $G_3$, then the vertices $a,v,x_1,x_2$ and $y$ induce a 4-fan~($F_4$) or a 4-wheel~($F_8$), depending on the presence of edge $av$.
\end{proof}

\begin{lemma}\label{lemma:G3-clique}
    If $G_3$ is a non-null complete subgraph and $y$ is a vertex of $G_3$, then every complete component of $G_1$ is either completely connected to or disconnected from $y$.
\end{lemma}
\begin{proof}
    Suppose that a complete component $A$ of $G_1$ is partially connected to a vertex $y$ in $G_3$. Then there exist a vertex $a \in A$ adjacent to $y$ and a vertex $b \in A$ not adjacent to $y$, and the vertices $a,b,x_1,x_2$ and $y$ induce a 4-fan~($F_4$).
\end{proof}

Let~$G_3$ be a complete subgraph. By Lemma~\ref{lemma:G3-connected-to-one}, at most one component from each~$G_1$ and~$G_2$ is connected to~$G_3$. 
In case such components exist, we denote the two respective components of~$G_1$ and~$G_2$ connected to~$G_3$ by~$A$ and~$V$. 
If no component in~$G_1$ (or~$G_2$) is connected to~$G_3$, then we assume that~$A$ or~$V$ is null. By Lemma~\ref{lemma:G3-clique} and Lemma~\ref{lemma:G3-connected-to-one}, a vertex of $G_3$ is either completely connected to $A$ or to $V$ or 
disconnected from $A$ or $V$. This yields the following partition of~$G_3$ into three complete subgraphs 
\[G_3 = G_3^A \vee G_3^V \vee G_3^R,\]
defined as follows.

\begin{definition}\label{definition:partitioning-G3}
By $G_3^A$ we denote the subgraph of~$G_3$ induced by the vertices in~$G_3$ adjacent to~$A$, by $G_3^V$ the subgraph of~$G_3$ induced by the vertices in~$G_3$ adjacent to $V$, and by $G_3^R$ the subgraph induced by the vertices in~$G_3$ 
only adjacent to~$x_1$ and~$x_2$.
\end{definition}
Note that $G_3^A$, $G_3^V$ and $G_3^R$ may be null graphs.

\begin{lemma}\label{lemmaZ}
Let $G_3$ be complete and $G_1$ contain a component $C$ of order at least $2$ as well as an additional vertex $d$ not in $C$. If $G_2$ and $G_3^A$ are non-null, then $G_3^A$ is completely connected to $C$ and disconnected from any other vertices in $G_1$.
\end{lemma}
\begin{proof}
Let $a,b$ be vertices in $C$, $v$ be a vertex in $G_2$ and $y$ be a vertex in $G_3^A$. By Definition~\ref{definition:partitioning-G3}, $G_3^A$~is completely connected to exactly one component in $G_1$ and disconnected from all others. We show 
that this component is $C$ by assuming, conversely, that $G_3^A$ is completely connected to $d$. By Lemma~\ref{lemma:all-but-one}, the vertex $v$ is completely connected to $d$ or $C$, or both. If $v$ is adjacent to $a,b$ and $d$, then the 
vertices $a,b,d,v,x_1,x_2$ and $y$ induce an $F_{15}$. Otherwise, if $v$ is not adjacent to $a$ or $d$, the vertices $a,d,v,x_2$ and~$y$ induce an $F_2, F_3$ or $F_7$. \end{proof}

\begin{lemma}\label{lemmaW}
    If $G_3$ is complete, $G_3^A$ is non-null and $G_1$ has at least two components, then any vertex in $G_2$ must be adjacent to all vertices in~$G_1$.
\end{lemma}
\begin{proof}
     Let $a,b$ be vertices from two components in $G_1$ and $v$ be a vertex in $G_2$. Without loss of generality, assume that $a$ is adjacent to $y \in G_3^A$. By Lemma~\ref{lemma:all-but-one}, $v$ is adjacent to $a$ or $b$. Hence if $v$ is 
not adjacent to~$a$, then the vertices $a,b,v,x_2$ and $y$ induce a $P_5$, and if $v$ is not adjacent to $b$, then the same vertices induce an $F_3$.
\end{proof}

If $G_3$ has a vertex not adjacent to any vertices in $G_1$ or $G_2$, (i.e.~$G_3^R$ is not null), we can say the following about $G_1$ and $G_2$.
\begin{lemma}\label{lemma:if-y-independent-one-comp-each}
    Let $G_3$ be complete and $G_3^R$ be non-null. Then $G_1$ and~$G_2$ are either null or complete.
\end{lemma}
\begin{proof}
    Let $y$ be a vertex in $G_3^R$. By definition, $y$ is not adjacent to any vertices in $G_1$ or $G_2$. Suppose $G_1$ has two non-adjacent vertices $a$ and $b$ and let $v \in G_2$. If neither~$a$ nor~$b$ is adjacent to~$v$, then the vertices $a,b,v,x_1$ and $x_2$ induce a chair~($F_1$). If only one vertex is adjacent to~$v$, then the vertices $a,b,v,x_2$ and $y$ induce an~$F_3$. Finally, if both~$a$ and~$b$ are adjacent to~$v$, then the same vertices induce a chair~($F_1$).
\end{proof}

\begin{lemma}\label{lemma:only-one-K2}
    Let $G_3$ be complete and $G_3^R$ be non-null. If both~$G_1$ and~$G_2$ contain more than one vertex, then $G_1$ and~$G_2$ are completely connected to each other.
\end{lemma}
\begin{proof}
    Let $y$ be a vertex in $G_3^R$. By Lemma~\ref{lemma:if-y-independent-one-comp-each}, $G_1$ and~$G_2$ are both complete. Let~$a,b$ be two vertices of~$G_1$ and~$v,w$ be two vertices of~$G_2$. We check two cases due to symmetry. If~$a,b$ and~$v,w$ are disconnected, then the vertices $a,b,v,w,x_1,x_2$ and $y$ induce an~$F_{11}$. If~$a$ is adjacent to~$v$ and not adjacent to~$w$, then the vertices $a,v,w,x_1$ and $y$ induce a~$P_5$.
\end{proof}

If $G_3^R$ is null but~$G_3$ is non-null and complete, we can state the following two lemmas.
\begin{lemma}\label{lemma:if-KV-not-null}
    Let $G_3$ be complete. If~$G_1$ contains a component $C$ of order at least $2$ and~$G_3^V$ is non-null, then $G_1$ and~$G_2$ each consist of a single component. The same holds if~$G_2$ contains a component of order at least $2$ and~$G_3^A$ 
is non-null.
\end{lemma}
\begin{proof}
    Let $a,b$ be two vertices in $C$ and $y \in G_3^V$. Then, by assumption, there exists a vertex $v \in G_2$ that is adjacent to $y$. First we show that~$G_1$ consists of a single component and this is complete by Lemmas~\ref{lemma:G3-not-clique} and~\ref{lemma:G3-if-N-then-independent}. Assume that~$c$ is a vertex from a second component in~$G_1$. We distinguish between three cases. If~$C$ is not connected to~$v$, then $c$ must be adjacent to~$v$ by Lemma~\ref{lemma:all-but-one} and the vertices~$a,b,c,v,x_1,x_2$ and~$y$ induce an~$F_{12}$. Secondly, if~$C$ is partially connected to~$v$, we can assume that~$a$ is adjacent to~$v$ whereas~$b$ is not. As~$c$ and~$v$ must be adjacent, the vertices~$a,b,c,v$ and~$y$ induce a chair ($F_1$). Lastly, if~$C$ and~$v$ are completely connected and~$c$ is adjacent to~$v$, then the vertices~$a,b,c,v,x_1,x_2$ and~$y$ induce an~$F_{14}$ while if~$c$ is not adjacent to~$v$, then the same vertices induce an~$F_{13}$.
    
    Now we show that~$G_2$ also consists of a single component. Assume that~$w$ is a vertex from a second component in~$G_2$. As~$y$ is adjacent to~$v$, it cannot be adjacent to~$w$ by Lemma~\ref{lemma:G3-connected-to-one}~(i). We show that the existence of~$w$ leads to a forbidden induced subgraph. Vertices~$a,b$ cannot both be adjacent to~$v$ and~$w$, otherwise the vertices~$a,b,v,w$ and~$x_1$ induce a split 3-star~($F_5$). Without loss of generality assume that it is vertex~$a$ that is not adjacent to both~$v$ and~$w$. If~$a$ is neither adjacent to~$v$ nor to~$w$, then the vertices~$a,v,w,x_1$ and $y$ induce an~$F_3$. If~$a$ is adjacent to~$v$ only, then the same vertices induce an~$F_7$. If~$a$ is adjacent to~$w$ only, then the same vertices induce a~$P_5$.
\end{proof}

\begin{lemma}\label{lemma:if-KR-null}
    If~$G_1$ has a component of order at least $2$, $G_2$ and~$G_3$ are non-null and $G_3$ is complete, then~$G_1$ and~$G_2$ each have at most two components.
\end{lemma}
\begin{proof}
    By Lemma~\ref{lemma:N-in-G1} and Lemma~\ref{lemma:at-most-two-components}, $G_2$ has at most two components. We now show the same for~$G_1$. Assume that~$G_1$ has three components and let~$a,b,c$ be vertices from these three components. Let~$v$ be a vertex in~$G_2$. By Lemmas~\ref{lemma:if-KV-not-null} and~\ref{lemma:if-y-independent-one-comp-each}, $G_3^V$ and $_3^R$ are null, i.e. $G_3 = G_3^A$. Let~$y_A$ be a vertex in~$G_3^A$ and without loss of generality let~$G_3^A$ be completely connected to~$a$. By Lemma~\ref{lemma:all-but-one}, no more than one vertex of~$a,b,c$ is non-adjacent to~$v$. If~$b$ and~$c$ are adjacent to~$v$, then the vertices $b,c,v,x_2$ and $y_A$ induce a chair~($F_1$). Otherwise we can assume that~$b$ is not adjacent to~$v$ and the vertices~$a,b,v,x_2$ and $y_A$ induce an~$F_3$.
\end{proof}

\section{Case distinctions}\label{section:case-distinctions}
In this section we establish that connected graphs without induced subgraphs $F_1$ to $F_{15}$ belong to one of the nine classes of graphs $E_1^\cup, E_2, \ldots, E_9$. Together with Theorem~\ref{thm:disconnected-gB-perfect-characterisation}, 
this concludes the proof of the implication 
(ii) $\Rightarrow$ (iii) of Theorem~\ref{thm:gB-perfect-characterisation}.

Again, let $G$ be a connected graph of order at least~2 without induced $F_1$ to $F_{15}$. Recall from Section~\ref{section:edge-decomposition} the dominating edge decomposition: $G$ has a dominating edge $x_1x_2$ and induced subgraphs $G_1, G_2$ and~$G_3$ as defined in 
Definition 13. Lemma~\ref{lemma:caseanalysis} consists of a series of case distinctions that collectively cover all possible structural configurations of these subgraphs. For each case, we prove that $G$ is an instance of one of the graph classes 
$E_1^\cup, E_2, \ldots, E_9$. The proof utilises the structural results developed in Section~\ref{section:edge-decomposition}.

\begin{lemma}\label{lemma:caseanalysis}
Let $G$ be a connected graph without induced $F_1$ to $F_{15}$. Assume we have a decomposition of $G$ into a dominating edge $x_1x_2$ and induced subgraphs $G_1, G_2$ and $G_3$ as above.
\begin{itemize}
\item[(a)] If $G_1$ or $G_2$ is null, then $G$ is an instance of the graph class $E_1$.
\end{itemize}

\noindent From now on, let neither $G_1$ nor $G_2$ be null. 

\begin{itemize}
\item[(b)] If $G_3$ is not complete, then $G$ is an instance of $E_2$.
\item[(c)] If $G_3$ is complete and $G_1$ or $G_2$ has a non-complete component $N$, then $G_3$ is null and $G$ is an instance of $E_3$.
\item[(d)] If $G_3$ is complete and $G_1$ and $G_2$ consist only of isolated vertices, then $G$ is an instance of $E_3$, $E_6$, $E_7$ or $E_9$.
\item[(e)] Let $G_3$ be complete. If one of $G_1$ and $G_2$ consists only of isolated vertices and the other consists of complete components only, at least one of them with cardinality at least $2$, then $G$ is an instance of $E_3$, $E_6$, $E_7$ or $E_8$.
\item[(f)] Let $G_3$ be complete. If $G_1$ and $G_2$ consist of complete components only and both contain at least one component of order at least $2$, then $G$ is an instance of $E_3$, $E_4$ or $E_5$. 
\end{itemize}
\end{lemma}

\stepcounter{proofcount}

\begin{proof}[of Lemma~\ref{lemma:caseanalysis} (a)]
    Without loss of generality assume that $G_2$ is null. By Lemma~\ref{lemma:N-structure}, $G_1$ consists of at most one non-complete component $N=K_r \vee (K_p \cup K_q)$ with $p,q,r \geq 1$ and any number of complete components. Our aim is to show that $G$ is in $E_1$. We distinguish between the case that $G_1$ has a non-complete component $N$ and the case that $G_1$ has no such component.

\begin{enumerate}[{Case }1:]
\item
Assume $G_1$ has a non-complete component $N$. By Lemmas~\ref{lemma:G3-not-clique} and~\ref{lemma:G3-if-N-then-independent}, $G_3$ must be complete and disconnected from $G_1$. It follows that $G$ is in $E_1$ with 
$H_0 := 
N$, $H_1 := G_3 \vee x_2$ and $H_2, \ldots, H_k$ consisting of the complete components of $G_1$.

\item
Next assume that $G_1$ has no non-complete component. $G_3$ may be complete or not complete. 

If $G_3$ is not complete, then it is disconnected from all components in $G_1$ by Lemma~\ref{lemma:G3-not-clique} and $G_3$ is a possibly degenerate ear graph, by Lemma~\ref{lemma:G3-shape}. It follows that $G$ is in $E_1$ 
with $H_0 := G_3 \vee x_2$ and $H_1, \ldots, H_k$ consisting of the components of $G_1$.
        
Now assume that $G_3$ is complete. By Lemma~\ref{lemma:G3-connected-to-one}~(i), only one component $A$ of $G_1$ can be adjacent to vertices in $G_3$. If no such $A$ exists, then $G$ is in $E_1$ with~$H_0$ null, $H_1 := G_3 \vee x_2$ and $H_2, 
\ldots, H_k$ consisting of the components of $G_1$. If $A$ does exist, it is completely connected to a subgraph $G_3^A$ of $G_3$ and disconnected from $G_3 \setminus G_3^A$ by Lemma~\ref{lemma:G3-clique} and Definition~\ref{definition:partitioning-G3}. It follows that $G$ is in $E_1$ with $H_0 := G_3^A \vee ( A \cup (x_2 \vee G \setminus G_3^A) )$ and $H_1, \ldots, H_k$ consisting of the components of $G_1$ except~$A$.
\end{enumerate}\nopagebreak~\end{proof}

\pagebreak[3]

From now on, assume that neither~$G_1$ nor~$G_2$ is null.

\begin{proof}[of Lemma~\ref{lemma:caseanalysis} (b)] Let $a$ and $v$ be vertices in $G_1$ and $G_2$, respectively. By assumption, $G_3$ is not complete and contains at least two non-adjacent vertices~$y_1$ and~$y_2$. By 
Lemma~\ref{lemma:G3-not-clique}, $y_1$ and $y_2$ are disconnected from $G_1$ and $G_2$. Next note that $G_1$ and $G_2$ are disconnected, otherwise we can assume that $a$ and $v$ are adjacent and thus the vertices~$a,v,x_1,y_1$ and~$y_2$ induce a chair~($F_1$). Lastly, both $G_1$ and $G_2$ contain exactly one vertex. Indeed, assume on the contrary that $G_2$ contains an additional vertex $w$. By the above, vertex $a$ is not adjacent to $v$ or $w$. If~$v$ and~$w$ are not adjacent, then the vertices~$a,v,w,x_1$ and~$x_2$ induce a chair~($F_1$). If~$v$ and~$w$ are adjacent, then the vertices~$a,v,w,x_1,x_2,y_1$ and~$y_2$ induce an~$F_{10}$. By Lemma~\ref{lemma:G3-shape}, $G_3$ is a possibly degenerate ear graph and it follows that $G$ is an instance of~$E_2$.
\end{proof}

From now on, assume that $G_3$ is complete (null or non-null).

\begin{proof}[of Lemma~\ref{lemma:caseanalysis} (c)] Without loss of generality, assume that $G_1$ has a non-complete component $N$. By Lemma~\ref{lemma:G3-if-N-then-independent}, $G_3$ is null. As $G_2$ is non-null by assumption, Lemma~\ref{lemma:N-in-G1} implies that $G_1$ consists entirely of the non-complete component $N = K_r \vee (K_1 \cup K_b)$ and $G_2$ consists of a single vertex that is completely connected to the $K_1$ and $K_b$ of $N$ and disconnected from $K_r$. It follows that $G$ is an~$E_3$.
\end{proof}

From now on, assume that neither $G_1$ nor $G_2$ has a non-complete component. By assumption and Lemma~\ref{lemma:one-component-of-size-2}, $G_1$ and $G_2$ are both non-null, contain no more than one component of order at least $2$ and any number of independent vertices. The three case distinctions (d) to (f) in Lemma~\ref{lemma:caseanalysis} distinguish between the three possibilities that neither $G_1$ and $G_2$ have a component of order at least $2$, that only $G_1$ has such a component and that both $G_1$ and $G_2$ have such a component.

\begin{proof}[of Lemma~\ref{lemma:caseanalysis} (d)] Assume that $G_1$ and $G_2$ consist only of $K_1$ components, i.e. isolated vertices. In particular, we distinguish between four cases that comprehensively cover all configurations of $G_1$ and $G_2$. We investigate how each case affects the structural possibilities for the subgraph $G_3 = G_3^A \vee G_3^V \vee G_3^R$ (see Definition~\ref{definition:partitioning-G3}). In Cases~\ref{case:two-or-more-1} to~\ref{case:two-or-more-3}, $G_1$ or $G_2$ consists of two or more components and, by Lemma~\ref{lemma:if-y-independent-one-comp-each}, $G_3^R$ is null.

\begin{enumerate}[{Case }1:]
\item
Assume $G_1 = K_1$ and $G_2 = K_1$. Then $G$ is in $E_6$ with $K_a := G_1$, $K_b := x_1 \vee G_3^A$, \mbox{$K_c := G_3^R$,} $K_d := x_2 \vee G_3^V$ and $e := G_2$. Subgraphs $G_1$ and $G_2$ may or may not be adjacent.

\item\label{case:two-or-more-1} Assume $G_1 = K_1$ and $G_2 = K_1 \cup K_1$. Let $a$ be the vertex in $G_1$ and $v,w$ be the two vertices in $G_2$. We distinguish between the two possibilities that $G_3^V$ is null or non-null. If $G_3^V$ is 
not null, 
then, by Lemma~\ref{lemma:G3-connected-to-one}~(i), we can assume without loss of generality that $v$ is connected to $G_3^V$ and $w$ is not.
Furthermore, by Lemma~\ref{lemmaW}, the vertex $a$ is adjacent to both $v$ and $w$ and $G$ is in $E_3$ with $K_n := x_1 \vee G_3^A$ and $K_m := G_3^V$. If $G_3^V$ is null, then by Lemma~\ref{lemma:all-but-one}, $a$ is adjacent to at least one 
of $v$ or $w$ and $G$ is in $E_7$ with $K_n := x_1 \vee G_3^A$.

\item\label{case:two-or-more-2} Assume $G_1 = K_1$ and $G_2 = K_1 \cup K_1 \cup K_1 \cup \dots \cup K_1$, where $G_2$ has at least three $K_1$ components. Let $a$ denote the vertex in $G_1$ and $u,v,w$ be three isolated 
vertices in 
$G_2$. We see that $G_3^V$ is null: if $G_3^V$ is not null, then, 
by Lemma~\ref{lemma:G3-connected-to-one}~(i), we can assume without loss of generality that $u$ is connected to $G_3^V$ and $v$ and $w$ are not. Furthermore,
by Lemma~\ref{lemmaW}, $a$ is adjacent to $u,v$ and $w$ and thus the vertices $a,v,w,x_1$ together with any vertex $y \in G_3^V$ induce a chair~($F_1$). As $a$ is adjacent to all but one vertex in $G_2$ by Lemma~\ref{lemma:all-but-one}, it 
follows 
that $G$ is in $E_7$ with $K_n := x_1 \vee G_3^A$.

\item\label{case:two-or-more-3} Assume $G_1$ and $G_2$ both consist of two or more $K_1$ components. Let $a,b$ be two vertices in $G_1$ and $v,w$ be two vertices in $G_2$. We see that $G_3^A$ and $G_3^V$ are null by assuming 
the 
contrary. If $G_3^A$ is non-null with $y \in G_3^A$ and $a$ is the neighbour of $y$ in $G_1$, then $a$ and $b$ are both adjacent to $v$ and $w$ by Lemma~\ref{lemmaW} and the vertices $b,v,w,y$ and $x_1$ induce a chair~($F_1$). By symmetry we see that $G_3^V$ is also null and it follows immediately that $G_3$ is null. By Lemma~\ref{lemma:all-but-one}, no more than one pair of vertices $(a,v)$ in $G_1 \times G_2$ is non-adjacent, so $G$ is in $E_9$. 
\end{enumerate}
\end{proof}

\begin{proof}[of Lemma~\ref{lemma:caseanalysis} (e)] Assume that $G_1$ contains a component of order at least $2$ and $G_2$ only contains $K_1$ components. By Lemma~\ref{lemma:at-most-two-components}, $G_2$ has exactly one or two $K_1$ components. We distinguish between these two cases which in turn also divide into various subcases.

\begin{enumerate}[{Case }1:]
\item Assume that $G_2$ consists of a single $K_1$ component denoted by vertex~$v$. We treat the two possibilities $G_1 = K_n$ and $G_1 = K_n \cup K_1 \cup \dots \cup K_1$, where $n \geq 2$, separately.

\begin{enumerate}[{Case \theenumi.}1:]
\item Assume that $G_1$ is complete, $G_1 = K_n$. If $G_3^R$ is not null, then $G_1$~and $G_2$ are either disconnected or completely connected: if vertices $a$ and $b$ in $G_1$ existed such that $a$ is adjacent to $v$ and $b$ is not, vertices 
$a,b,v,x_2$ and any vertex in $G_3^R$ would induce a $P_5$. It follows that $G$ is in $E_6$ (with $K_a=G_1$, $K_b=G_3^A\cup x_1$, $K_c=G_3^R$, $K_d=G_3^V\cup x_2$, and $e=G_2$). If $G_3^R$ is null, then $G_1$ and $G_2$ can have any adjacency 
relation and by Lemma~\ref{lemma:Kn-to-Kn}, $G$ is in $E_5$ with $A_1 := x_1 \vee G_3^A$ and $Z_1 := x_2 \vee G_3^V$.

\item Secondly, assume that $G_1$ contains one or more $K_1$ components in addition to the $K_n$, so $G_1 = K_n \cup K_1 \cup \dots \cup K_1$. Let $a$ be a vertex in the $K_n$ and $c$ be a vertex in a $K_1$ of $G_1$. By  Lemmas~\ref{lemma:if-y-independent-one-comp-each} and~\ref{lemma:if-KV-not-null}, $G_3^R$ and $G_3^V$ are null. $G_3^A$~can be null or non-null and we examine both possibilities.

First let $G_3^A$ be non-null and $y \in G_3^A$. 
Then we have $G_1 = K_n \cup K_1$ by Lemma~\ref{lemma:if-KR-null}. Furthermore, by Lemma~\ref{lemmaZ}, $G_3^A$ is completely connected to $K_n$ and disconnected from~$c$. Vertex $v$ is completely connected to $G_1$, otherwise the vertices $a,c,v,x_2$ and~$y$ induce an $F_3, F_7$ or $P_5$. Hence $G$ is in $E_3$ (with $K_m=G_3^A$, $K_n=K_n$, $a=v$, $b=x_2$, $c=x_1$, $d=c$).

Next let $G_3^A$ be null and let $C$ denote the $K_n$ in $G_1$. Then $C$ and $G_2$ must be disconnected or completely connected. Indeed, assume to the contrary that the vertex $v$ of $G_2$ is adjacent to $a$ and not adjacent to $b$ in $C$. Then vertex $c$ is adjacent to $v$ by Lemma~\ref{lemma:all-but-one} and the vertices $a,b,c,v$ and $x_2$ induce a chair~($F_1$).

If $C$ and $G_2$ are disconnected, then all $K_1$ components in $G_1$ are connected to $G_2$ by Lemma \ref{lemma:all-but-one} and $G$ is in $E_7$. If $C$ and $G_2$ are completely connected, then at most one $K_1$ component in $G_1$ is 
disconnected from $G_2$ by Lemma~\ref{lemma:all-but-one} and $G$ is also in $E_7$.
\end{enumerate}

\item Now assume that $G_2 := K_1 \cup K_1$ has two $K_1$ components, denoted by vertices $u$ and $v$. By Lemmas~\ref{lemma:if-y-independent-one-comp-each} and~\ref{lemma:if-KV-not-null}, $G_3^R$ and $G_3^V$ are null and $G_3^A$ may be null 
or non-null. Without loss of generality assume by Lemma~\ref{lemma:all-but-one} that vertex $u$ is completely connected to $G_1$. As in Case 1, $G_1$~may have the structure $G_1 = K_n$ or $G_1 = K_n \cup K_1 \cup \dots \cup K_1$, where 
$n \geq 2$.
 
\begin{enumerate}[{Case \theenumi.}1:]
\item Assume $G_1 = K_n$ with $n \geq 2$. By Lemma~\ref{lemmaY}, vertex $v$ is adjacent to no more than one vertex $\tilde{a}$ in the $K_n$. It follows that $G$ is in $E_3$ (with 
$K_m=G_1\setminus\tilde{a}$, 
$K_n=G_3^A\cup x_1$, 
$a=x_2$, 
$b=u$, $c=\tilde{a}$, and $d=v$).

\item  Assume that $G_1$ has one or more $K_1$ components in addition to the $K_n$. Let $C$ denote the component of order at least $2$ in $G_1$ and $d$ denote a $K_1$ component in $G_1$. 
First we note that, by Lemma~\ref{lemmaY}, vertex $v$ is adjacent to at most one vertex in $C$. Hence by Lemma~\ref{lemma:all-but-one}, the vertex $v$ is completely connected to every $K_1$ component of $G_1$. It follows that $v$ is 
disconnected from $C$: otherwise, if $v$ is adjacent 
to $a \in C$ and not adjacent to $b \in C$, the vertices $a,b,d,v$ and $x_2$ would induce a chair~($F_1$). 
Next we see that 
$G_3^A$ must be null. Assume to the contrary that $y$ is a vertex in $G_3^A$ and let $b \in C$. By Lemma~\ref{lemma:all-but-one}, vertex $d$ is adjacent to $u$ and $v$ and by Lemma~\ref{lemmaZ}, $y \in G_3^A$ is adjacent to~$b$ and not 
to~$d$. It follows that the vertices $b,d,v,x_2$ and $y$ induce a $P_5$. In conclusion, we have $G \in E_8$.
\end{enumerate}
\end{enumerate}
\end{proof}

\begin{proof}[of Lemma~\ref{lemma:caseanalysis} (f)] Assume that both $G_1$ and $G_2$ contain a component of order at least $2$. By Lemma~\ref{lemma:at-most-two-components}, $G_1$ and $G_2$ have at most one additional $K_1$ component each. We distinguish between three cases due to symmetry.

\begin{enumerate}[{Case }1:]
\item Assume $G_1 = K_n$ and $G_2 = K_m$ with $m,n \geq 2$. If $G_3^R$ is not null, then $G_1$ and $G_2$ are completely connected, by Lemma~\ref{lemma:only-one-K2}, and $G$ is in $E_5$. If $G_3^R$ is null, we see with the help of Lemma~\ref{lemma:Kn-to-Kn} that $G$ is also in $E_5$.

\item Assume $G_1 = K_n$ and $G_2 = K_m \cup K_1$ with $m,n \geq 2$. Let $v$ denote the $K_1$ in $G_2$. By Lemma~\ref{lemma:if-y-independent-one-comp-each}, $G_3^R$ is null and by Lemma~\ref{lemma:if-KV-not-null}, all of $G_3$ is null. 
Lemma~\ref{lemmaX} shows that $G_1$ and the $K_m$ in $G_2$ are either disconnected or completely connected. If $G_1$ and the $K_m$ are disconnected, then vertex $v$ is completely connected to $G_1$ by Lemma~\ref{lemma:all-but-one} and $G$ is 
in $E_4$. If $G_1$ and the $K_m$ are completely connected, then vertex $v$ cannot be connected to more than one vertex in $G_1$, by Lemma~\ref{lemmaY}, and $G$ is in~$E_3$.

\item Assume $G_1 = K_n \cup K_1$ and $G_2 = K_m \cup K_1$ with $m,n \geq 2$. Let $a$ and $C$ denote the $K_1$ and $K_n$ in $G_1$, respectively, and let $v$ and $W$ denote the $K_1$ and $K_m$ in $G_2$. By 
Lemmas~\ref{lemma:if-y-independent-one-comp-each} and~\ref{lemma:if-KV-not-null}, $G_3$~is null. Lemma~\ref{lemmaX} indicates that $C$ and $W$ are either disconnected or completely connected. We see that $C$ and $W$ are disconnected. Indeed, if they are completely connected, Lemma~\ref{lemmaY} says that vertex $a$ cannot be completely connected to $W$ and vertex $v$ cannot be completely connected to $C$, contradicting Lemma~\ref{lemma:all-but-one}. Hence, as $C$ and $W$ are thus disconnected, Lemma~\ref{lemma:all-but-one} states that vertices $a$ and $v$ are completely connected to $G_2$ and $G_1$, respectively, and $G$ is in $E_4$.
\end{enumerate}
\end{proof}

\pagebreak[4]

\section{Strategies for Alice on $E$}\label{section:strategies}
In this section we show that Alice has a winning strategy on each $G \in E$ with $\omega(G)$ colours. As the class $E$ is hereditary by Lemma~\ref{lemma:E-is-hereditary} (that is, $H \in E$ for any induced subgraph $H$ of $G$), this proves the implication (iii) $\Rightarrow$ (i) of Theorem~\ref{thm:gB-perfect-characterisation}.

The strategies presented here are not `colourblind' but rely on Alice being able to react to specific colours that Bob uses. An effort has been made to generalise the strategies as much as possible. In particular, some of Alice's moves are derived from strategies for small base graph structures discussed in Section~\ref{subsec:strategies-preparations}. However, the main parts of the strategies (described in Sections~\ref{subsec:stratA}--\ref{subsec:stratB}) are specific to each graph class. To the best of our knowledge, it is not possible to unify these strategies further.

\subsection{Preparations}\label{subsec:strategies-preparations}
A subgraph is \emph{fully coloured} if all its vertices are coloured, \emph{uncoloured} if none of its vertices are coloured and \emph{partially coloured} otherwise. A vertex is \emph{critical} if it is uncoloured and the sum of 
the number of adjacent colours and the number of uncoloured neighbours is at least $\omega(G)$, and \emph{safe} otherwise. A subgraph is \emph{critical} if it contains a critical vertex, otherwise it is \emph{safe}. We make a number of observations.

\begin{observation}\label{observation:saving}
	A vertex is made safe either by colouring it or by reducing the number of uncoloured neighbours to less than the difference of $\omega(G)$ and the number of adjacent colours.
\end{observation}

\begin{observation}\label{observation:terminal}
	Once a vertex is safe, it remains so.
\end{observation}

\begin{observation}\label{observation:winning}
	Alice wins once all vertices are safe.
\end{observation}

\begin{observation}\label{observation:fixed}
	Alice wins once all uncoloured vertices in the partially coloured graph only admit a single colour and these colours together with the coloured vertices yield a proper colouring.
\end{observation}

Here, we refer to $K_{1,k}$ as \emph{$k$-stars}.
We present simple strategies for $k$-stars, ear graphs, the bull and the dragon graph, shown in Figure~\ref{fig:simple-strategies}. These strategies constitute the building blocks from which we construct the strategies for each $E_i$. In the description of each strategy we refer to the vertex labels provided in the figure.

\begin{figure}[htb] \subfigure[$k$-star ($K_{1,k}$)\label{fig:pre-k-star}]{\includegraphics{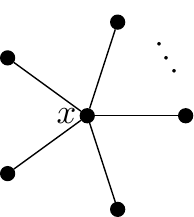}}\hfill
 \subfigure[ear graph ($K_a\vee(K_b\cup K_c)$)\label{fig:pre-ear-graph}]{\includegraphics{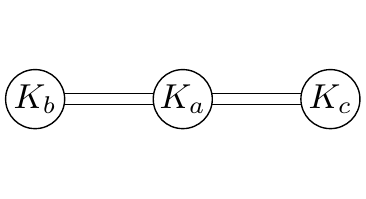}}\hfill
 \subfigure[bull\label{fig:pre-bull}]{\includegraphics{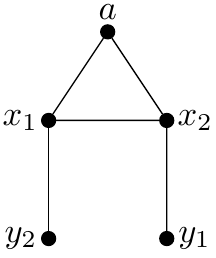}}\hfill
 \subfigure[dragon\label{fig:pre-dragon}]{\includegraphics{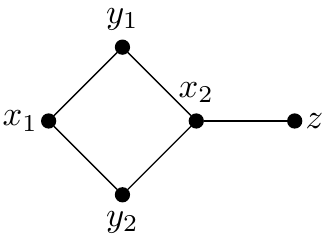}}
	\caption{We provide simple strategies for these four graphs.}
	\label{fig:simple-strategies}
\end{figure}

\begin{lemma}[$k$-star strategy]\label{lemma:k-star-strategy}
	Alice wins the game $g_B$ on a $k$-star with two colours.
\end{lemma}
\begin{proof}
	Note that $x$ is the only critical vertex. If Bob colours $x$ in his first move, then all vertices are safe and Alice wins. Otherwise Alice colours $x$ in her first move and wins.
\end{proof}

\begin{lemma}[ear graph strategy]\label{lemma:ear-strategy}
	Alice wins the game $g_B$ on an ear graph $G$ with $\omega(G)$ colours.
\end{lemma}
\begin{proof}
	Assuming without loss of generality that $b \leq c$, we have $\omega(G) = a+c$, so only $K_a$ is critical initially. By Observation~\ref{observation:saving}, we can make the vertices of $K_a$ safe either by colouring them or by ensuring that every 
vertex in $K_b$ has a vertex in $K_c$ with the same colour, leading to the following simple strategy for Alice. If Bob colours a vertex in $K_b$ or $K_c$, Alice applies the same colour to a vertex in the other clique. If $K_b$ is now fully coloured, Alice wins. If Bob colours a vertex in $K_a$, Alice attempts to colour another vertex in $K_a$. If the clique $K_a$ is now fully coloured, Alice wins.
\end{proof}


\begin{lemma}[bull strategy]\label{lemma:bull-strategy}
Alice wins the game $g_B$ on a bull with three colours.
\end{lemma}
\begin{proof}
At the start, vertices $x_1, x_2$ are critical. If Bob colours vertex $x_i$ or $y_i$, Alice applies the same colour to the other vertex and wins. If Bob instead colours vertex $a$, Alice applies the same colour to $y_2$, which makes 
$x_1$ safe. If Bob now colours $x_2$, Alice wins, otherwise she colours $x_2$ herself to win.
\end{proof}

\begin{lemma}[dragon strategies]\label{lemma:dragon-strategy}
Alice wins the game $g_B$ on a dragon with two colours.
\end{lemma}
\begin{proof}
In the beginning, vertices $x_1, x_2, y_1, y_2$ are critical. Note that Alice can win in one round by ensuring one of three possibilities: $x_1, x_2$ are coloured the same, $x_2, y_i$ are coloured differently or $x_1, z$ are coloured 
differently. In all three cases Bob can no longer win. Hence if Bob colours a vertex in one of these three pairs, Alice's approach is to colour the other vertex in the appropriate colour. As $x_1$ and $x_2$ appear in multiple pairs, any strategy using this approach is well-defined only once we define an order in which Alice tries to complete a pair colouring. In the \emph{pairing dragon strategy}, this order is given by $(x_1, x_2), (y_i, x_2), (z, x_1)$ while the 
order in the \emph{bipartite dragon strategy} is given by $(y_1, x_2), (y_2, x_2), (z, x_1)$. \end{proof}

Using these preparations, Sections \ref{subsec:stratA} -- \ref{subsec:stratB} contain winning strategies for Alice on each graph $G\in E$ with $\omega(G)$ colours.

\subsection{Strategies for Alice on the base structures ($E_1^\cup$ and $E_5$)}\label{subsec:stratA}
We prove that Alice has a winning strategy on each graph $G \in E$ with $\omega(G)$ colours. 
First we discuss the most general structures $E_1^\cup$ and $E_5$, which we call \emph{base structures}, and then the other seven, 
more specialised graph classes. The strategy on $E_1^\cup$ combines the $k$-star and the ear graph strategies. The strategy on $E_5$ follows a pairing argument similar to the pairing argument in the ear graph strategy.

\begin{proposition}[$E_1^\cup$]\label{proposition:E1}
Alice wins on $G \in E_1^\cup$ with $\omega(G)$ colours.
\end{proposition}
\begin{proof}
We can prove this proposition non-constructively as follows. As the graph $G \cup G$ is an instance of $E_1^\cup$ with at least two connected components, it is $g_B$-perfect by Theorem~\ref{thm:disconnected-gB-perfect-characterisation}. This implies that Alice can win on its induced subgraph $G$ with $\omega(G)$ colours.

Alternatively, we provide explicit strategies for Alice to win on graphs in $E_1$ and $E_1^\cup$. Suppose first that $G \in E_1$ and assume without loss of generality that $b \leq c$. Then $\omega(G) = 1 + \max \{ a+c, |H_1|, \ldots, |H_k| \}$. If $a + c < |H_i|$ for some $H_i$ or $b=0$, then only $x$ is critical and Alice can follow the same strategy as on the $k$-star (see Lemma~\ref{lemma:k-star-strategy}). If $a+c \geq |H_i|$ for all $H_i$ and $b\ge1$, then the vertices of $K_a$ and $x$ are critical. In this case we combine the $k$-star and ear graph strategies (cf.~Lemmas~\ref{lemma:k-star-strategy} and~\ref{lemma:ear-strategy}) as follows.

\pagebreak[3]

\begin{enumerate}
	\item While neither $K_a$ nor $K_b$ is fully coloured:
	\begin{enumerate}[(i)]
		\item If Bob colours $x$, Alice colours a vertex in $K_a$.
		\item If Bob colours a vertex in $K_b$ or $K_c$, Alice applies the same colour to a vertex in the other clique.
        \item Suppose Bob colours a vertex in $K_a$ or in some $H_i$. If $x$ and $K_a$ are now fully coloured, Alice wins by Observation~\ref{observation:winning}. If $x$ is uncoloured, Alice colours it. If $x$ is coloured but there is an uncoloured vertex in $K_a$, Alice colours such a vertex.
	\end{enumerate}	
	\item Suppose $K_a$ or $K_b$ is fully coloured. If $x$ is coloured, Alice wins. Otherwise she ensures on her next move that $x$ is coloured and wins.
\end{enumerate}


Now suppose $G \in E_1^\cup$ with at least two connected components. Alice always responds to Bob's move by colouring a vertex in the same component as Bob according to the above strategy, unless Bob just coloured the last vertex in a component. We consider the latter case in more detail. If $x$ and $K_a$ are safe in every other component of $G$, Alice wins. Otherwise there exists a component with an uncoloured $x$ or an uncoloured vertex in $K_a$. If there is an uncoloured~$x$, she colours~$x$. Otherwise she colours an uncoloured vertex in~$K_a$.

Observe that the strategies for $E_1$ and $E_1^\cup$ are correct since, after every move by Alice, the following invariants hold.
\begin{enumerate}
\item Every vertex in $K_b$ has the same colour as a vertex in $K_c$, as in the ear graph strategy. 
\item The central vertex is uncoloured only if $K_a$ is not fully coloured and the $H_i$ are completely uncoloured. Together with 1.~, this guarantees that $x$ can be coloured later, as in the star strategy.
\end{enumerate}
\end{proof}

\begin{proposition}[$E_5$]\label{proposition:E5}
Alice wins on an expanded cocobi $G \in E_5$ with
\[
	\max\{|A_R|+\sum_{i=1}^k|A_i|,|Z_R|+\sum_{i=1}^k|Z_i|,\max\limits_{i=1,\ldots,k}\{|A_i|+|Z_i|\}\}
\]
colours.
\end{proposition}

\begin{proof}
Denote
\[
	A:=A_R\cup\bigcup_{i=1}^kA_i,\quad Z:=Z_R\cup\bigcup_{i=1}^kZ_i,
\]
and assume without loss of generality that $|A| \geq |Z|$ and $|A_1|+|Z_1| \geq |A_i|+|Z_i|$ for all $1 \leq i \leq k$. 
We show that Alice can win on $G$ with $\omega(G)$ colours. For this, we distinguish between two cases.

\pagebreak[3]

\begin{caseEE} 
$|A_1| + |Z_1| > |A|$.
\end{caseEE}

In this case, the clique number of $G$ is $\omega(G) = |A_1| + |Z_1|$. 
Observe that
\begin{align}
	|A_1| &> |A|-|Z_1| \geq |Z| - |Z_1|, \label{eq:parta}\\
	|Z_1| &> |A| - |A_1|.\label{eq:partb}
\end{align}


The following instructions suffice for Alice to win on $G$ with $\omega(G)$ colours. We give Alice's moves when Bob colours a vertex in $Z_1$ or $A\setminus A_1$. By symmetry, analogous instructions hold for $A_1$ and $Z \setminus Z_1$. 


\begin{enumerate}
\item Suppose Bob colours a vertex in $Z_1$.
\begin{enumerate}[(i)]
\item If $A \setminus A_1$ is not fully coloured, Alice applies the same colour to a vertex in $A \setminus A_1$.
\item If $A \setminus A_1$ is fully coloured and $Z_1$ is not, Alice colours another vertex in $Z_1$, as $A_1$ and $Z_2, \ldots, Z_k$ are safe. 
\item If $A\setminus A_1$ and $Z_1$ are fully coloured, every vertex is safe and Alice wins.\nopagebreak
\end{enumerate}
\item If Bob colours a vertex in $A\setminus A_1$, Alice applies the same colour to a vertex in $Z_1$. Such a vertex exists due to~\eqref{eq:partb} (or~\eqref{eq:parta} for the symmetric case).
\end{enumerate}


\begin{caseEE} 
$|A_1| + |Z_1| \leq |A|$. 
\end{caseEE}

In this case, $\omega(G) = |A|$ and we have exactly~$|A|$ colours 
available in the game. Since $G$ is the complement of a bipartite graph with bipartition $(A,Z)$, the graph $G$ is perfect (cf.~\citet{lovasz}). In particular, G has a proper vertex colouring with $|A| = \omega(G)$ colours. Moreover, every 
colour class consists of at most two vertices. If a vertex~$a$ in~$A$ has a vertex $v$ with the same colour in~$Z$, we say that~$a$ and~$v$ are \emph{paired}. Otherwise we call the vertex \emph{unpaired}. Since every colour appears in $A$, every vertex of $Z$ is paired. The following modification of the vertex colouring results in another valid vertex colouring: Let~$a$ be an unpaired vertex in~$A_i$ (or in~$A_R$) and let~$v$ be a vertex in~$Z \setminus Z_i$ (or in~$Z$) that is paired with~$b \in A$, respectively. We can pair~$a$ and~$v$ by assigning~$v$ the colour of~$a$. This means~$v$ and~$b$ are no longer coloured the same, and~$b$ is now an unpaired vertex. Alice wins using the following strategy.
\begin{enumerate}
\item If Bob colours a paired vertex, Alice applies the same colour to the other vertex of the pair. If $Z$ is then fully coloured, Alice wins.
\item Suppose Bob colours an unpaired vertex~$a$ in~$A_i, 1 \leq i \leq k$ [or $A_R$]. By rule 1, he uses a new colour.
	\begin{enumerate}[(i)]
	\item If $Z \setminus Z_i$ [or $Z$] is not fully coloured, Alice applies the same colour to any other vertex~$v$ in~$Z \setminus Z_i$ [or $Z$]. This pairs up vertices~$a$ and~$v$ and unpairs the vertex that~$v$ was originally paired with. 
	\item If $Z \setminus Z_i$ is already fully coloured, both $Z \setminus Z_i$ and $A \setminus A_i$ are safe and any uncoloured vertex in $A_i$ is an unpaired vertex. 
		\begin{enumerate}[(a)]
		\item If there is an uncoloured vertex in $A_i$, Alice colours another vertex in $A_i$. 
		\item If all vertices in $A_i$ are coloured, the only uncoloured vertices are in $A\setminus A_i$ and $Z_i$. Thus every vertex is safe and Alice wins. 
		\end{enumerate}
	\end{enumerate}
\end{enumerate}

Note that Bob can only win if a pair is coloured in two different colours and there is no possibility to unpair it. However, Alice's strategy avoids this situation, since the following invariants hold after every move by Alice.
\begin{enumerate}
\item The two vertices in each pair are either both coloured or both uncoloured. This can be seen as a generalisation of the pairing argument in the ear graph strategy given in Lemma~\ref{lemma:ear-strategy}.
\item Unpaired vertices are only coloured once all paired vertices are coloured.
\end{enumerate}
This proves the correctness of the strategy described above.
\end{proof}

\subsection{Strategies for Alice on the bipartite structures ($E_9$, $E_7$ and $E_8$)}\label{subsec:stratD}
It is known \citep[Corollary 20]{andres4} that Alice wins on $G \in E_9$ with $\omega(G)$ colours. In our proof of Proposition~\ref{proposition:E9}, we give an explicit strategy, the \emph{bipartite strategy}, that also underlies 
the strategies for $E_7$ and $E_8$ below. 

\begin{proposition}[$E_9$]
\label{proposition:E9}
	Alice wins on an (almost) complete bipartite graph $G \in E_9$ with $\omega(G)$ colours.
\end{proposition}
\begin{proof}
Let $G \in E_9$ be a complete bipartite graph $K_{m,n}$ or almost complete bipartite graph $K_{m,n}-uv$, for some edge $uv$. Further, let $(U,V)$ be the vertex bipartition of $G$ with $|U|=m$ and $|V|=n$. We assume without loss of generality that $\min\{m,n\} \geq 2$, or else the graph $G$ is in $E_1^\cup$. Hence $\omega(G)=2$ and $G$ is connected. The following strategy generalises the bipartite dragon strategy (see Lemma~\ref{lemma:dragon-strategy}). 

Suppose first that $G$ is complete bipartite. Then if Bob colours a vertex in $U$ or $V$, Alice applies the second colour to a vertex in the other vertex set and wins by Observation~\ref{observation:fixed}. Now suppose that $G$ is almost complete bipartite and $uv$ is the missing edge in $G$. If Bob colours $u$ or $v$, Alice applies the second colour to the other vertex. If Bob colours a vertex in $U-u$ or $V-v$, Alice applies the second colour to a vertex in the other vertex set. Again Alice wins by Observation~\ref{observation:fixed}.
\end{proof}

In order to give strategies for Alice to win on $E_7$ and $E_8$, we effectively follow the bipartite strategy and the pairing dragon strategy at the same time. 

\begin{proposition}[$E_7$]\label{beweisEsieben}
Alice wins on $G \in E_7$ with $\max\{2,n+1\}$ colours.
\end{proposition}

\begin{proof}
Let $m$ be the number of vertices in the bottom row. Without loss of generality, we assume that $\min\{m,n\} \geq 2$, or else $G$ is an instance of $E_5$ or $E_9$. 
Note that the vertices of the bottom row are safe, since $\omega(G) = n+1 \geq 
3$. Recall that $c$ is completely connected to $b$ or $K_n$ (cf.~Section~\ref{section:graph-classes}).

First assume that $c$ and $K_n$ are disconnected. Then only $a$ and $c$ are critical and Alice wins in at most two moves, as follows. If Bob colours $a$ or $c$, Alice applies the same colour to the other vertex and wins by Observation~\ref{observation:winning}. Otherwise, Alice colours $a$ in her first move and ensures that $c$ is coloured after her second move. Since at most two colours have been used for neighbours of $c$ before Alice's second move and $\omega(G) \geq 3$, she can use the third colour for $c$ and win by Observation~\ref{observation:winning}.

Now suppose that $K_n$ and $c$ are completely connected. Then the vertices $a$ and $c$ as well as the clique $K_n$ are critical. Note that every vertex is safe once $a$ and $c$ have the same colour and Alice wins by Observation~\ref{observation:winning}. On the other hand, if they are coloured differently, Alice loses. Alice's strategy consists in forcing $a$ and $c$ to receive the same colour. We divide Alice's strategy into two phases. Throughout, she makes sure that the colours of the bottom row form a subset of the colours in $K_n$. 

\paragraph{Phase 1:} The first phase consists of up to $\min\{n-1,m \}$ rounds with the following rules. 
\begin{enumerate}
	\item If Bob colours $a$ or $c$, Alice wins by applying the same colour to the other vertex.
	\item If Bob colours a vertex in $K_n$, Alice colours a vertex of the bottom row, and vice versa. While doing so, she makes sure to choose a colour so that the bottom row only contains colours also found in $K_n$.
\end{enumerate}

If $a$ and $c$ are not yet coloured at the end of Phase 1, Alice proceeds to Phase 2.

\paragraph{Phase 2:} As in Phase 1, Alice responds to Bob colouring $a$ or $c$ by applying the same colour to the other of the two vertices and wins by Observation~\ref{observation:winning}. For her other responses to Bob's moves, we distinguish between three cases.

\begin{enumerate}[{Case }1:]
\item Suppose $m < n$. Then the bottom row is fully coloured and $K_n$ contains all its colours. If Bob colours a vertex in $K_n$, Alice colours another vertex in $K_n$ unless the clique is fully coloured, in which case $a$ and $c$ are the only uncoloured vertices left and only admit a single colour~$n+1$, allowing Alice to win by Observation~\ref{observation:fixed}.
\end{enumerate}

For the remaining two cases, we note that colours $1$ to $n-1$ have been used so far to colour $K_n$ and the bottom row. Clearly, the single remaining uncoloured vertex in $K_n$ only admits colour $n$.

\begin{enumerate}[{Case }1:]\stepcounter{enumi} 
\item Suppose $m = n$. Then $K_n$ and the bottom row each contain one uncoloured vertex. If Bob colours the vertex in $K_n$ in colour $n$, Alice applies the same colour to the uncoloured vertex in the bottom row. Conversely, if Bob colours the vertex in the bottom row, Alice colours the vertex in $K_n$ in colour $n$. As in Case 1, Alice wins by Observation~\ref{observation:fixed}.

\item Suppose $m > n$. Then to begin with, $K_n$ has one uncoloured vertex and the bottom row has at least two. We can assume that the colours used so far are $1, \ldots, n-1$, and $a$ and $b$ are still uncoloured. 

	\begin{enumerate}[1.]
	\item If Bob colours the uncoloured vertex in $K_n$, Alice colours $a$ with colour~$n+1$. Hence, as colour $n+1$ is not feasible for the bottom row any more and $c$ is guaranteed to admit colour~$n+1$, Alice wins.

	\item If Bob colours $b$ in colour $n$, Alice applies colour~$n+1$ to $c$ and wins.

	\item If Bob colours any other uncoloured vertex in the bottom row in colour $n$, Alice applies colour~$n+1$ to $a$ and wins. 

	\item Suppose Bob colours a vertex of the bottom row with a colour from $\{1,\ldots,n-1\}$. If the bottom row is fully coloured, Alice wins by colouring the last vertex in $K_n$. If the bottom row has an uncoloured vertex, Alice applies the same colour to it. 
	\end{enumerate}
\end{enumerate}

In all three cases, Alice maintains the invariant that the bottom row only contains colours also found in~$K_n$. This implies that the bottom row only ever contains colours $1$ to $n$, so that $a$ and $c$ admit \mbox{colour~$n+1$}. Hence 
once all the vertices are coloured, $a$ and $c$ are guaranteed to be safe.
\end{proof}

The following proposition is proved in a similar way.

\begin{proposition}[$E_8$]\label{beweisEacht}
Alice wins on $G\in E_8$ with $\max\{2,n+1\}$ colours.
\end{proposition}

\begin{proof}
Without loss of generality, we assume that $n \geq 2$, or else $G \in E_9$. Moreover, we assume that $m \geq 2$, as $G \in E_1^\cup$ for $m=0$ and $G \in E_5$ for $m=1$. Note that for $n=2$, every vertex is critical and for $n > 2$, vertices $a$, $c$, $d$ and the clique $K_n$ are critical. If $a$ and $c$ have the same colour, every vertex apart from $d$ is safe. On the other hand, if $a$ and $c$ have different colours, Alice loses. Therefore Alice's strategy focuses on forcing $a$ and $c$ to have the same colour and making $d$ safe. Hence we follow the same two-phase strategy as for $E_7$ with minor changes to accommodate for vertex $d$.

\paragraph{Phase 1:} The first phase consists of up to $\min\{n-1,m\}$ rounds with the following rules. 

\begin{enumerate}
\item If Bob colours $a$ or $c$, Alice applies the same colour to the other vertex. As the only remaining critical vertex is $d$ and the bottom row contains at most $n-1$ colours so far, Alice wins on her next move by colouring $d$ with colour $n+1$. 

\item If Bob colours a vertex in $K_n$, then Alice colours a vertex of the bottom row, and vice versa. While doing so, she makes sure to choose a colour so that the bottom row only contains colours also found in $K_n$.

\item Suppose Bob colours $d$ with some colour $\alpha$. Note that there are now at least two uncoloured vertices in $K_n$, prior to Alice’s response, as Phase 1 has at most $n-1$ rounds and at most one vertex in $K_n$ is coloured in each round. Hence, if $\alpha$ does not already appear as a colour in $K_n$, Alice colours a vertex in $K_n$ with $\alpha$, otherwise she colours a vertex in $K_n$ with a new colour.
\end{enumerate}

\paragraph{Phase 2:} Suppose Alice has not won in Phase 1 and vertices $a$ and $c$ remain uncoloured. There are two cases which we consider separately for the second phase. We can assume that the colours used so far are $1, \ldots, \min\{n-1,m\}$.

\begin{enumerate}[{Case} 1:]
\item Suppose the bottom row is fully coloured. Then $K_n$ has at least one uncoloured vertex and $d$ is safe. If Bob colours $a$ or $c$, Alice wins by applying the same colour to the other. If Bob colours any other vertex, Alice wins by colouring a vertex in $K_n$ until the clique is fully coloured.

\item Suppose the bottom row is not fully coloured. In this case we have $\min\{n-1,m \} = n-1$ and the strategy of Phase~1 implies that $K_n$ has exactly one uncoloured vertex.

	\begin{enumerate}[1.]
	\item If Bob colours $d$, Alice colours $a$ with colour $n+1$, guaranteeing the feasibility of colour~$n+1$ for vertex $c$ forever, and winning in the process. 

	\item If Bob colours $a$ or $c$, Alice applies the same colour to the other vertices. If, at this point, $d$ is already coloured, she wins immediately, otherwise she wins on her next move by colouring $d$ with the same colour as $a$ and $c$. 

	\item If Bob colours the vertex in $K_n$ or a vertex in the bottom row with colour~$n$, Alice colours $a$ with colour $n+1$, ensuring that no vertex in the bottom row admits $n+1$ any more. This makes vertices $a,c$ and $d$ safe and Alice wins. 

	\pagebreak[3]
	\item Suppose Bob colours a vertex in the bottom row using an existing colour from the row. We can assume that $K_n$ is still not completely coloured, otherwise Alice would have won already. If Bob fills the bottom row with his move, Alice colours the last uncoloured vertex in $K_n$ and wins as $a,c$ and $d$ are now safe. If he does not fill the row, Alice applies the same colour to another vertex in the row. If the row is now fully coloured, Case 1 applies.
	\end{enumerate}
\end{enumerate}
\end{proof}

\subsection{Strategies for Alice on the bull structures ($E_2$ and $E_6$)}

We call $E_2$ and $E_6$ \emph{bull structures}, as the strategies described below are (partially) based on the bull strategy (Lemma~\ref{lemma:bull-strategy}).

\begin{proposition}[$E_2$]\label{proposition:E2}
Alice wins on $G \in E_2$ with $\omega(G) = 2+a+\max\{b,c\}$ colours.
\end{proposition}
\begin{proof}
	Without loss of generality, assume $b \leq c$. At the start, vertices $x_1, x_2$ and the clique $K_a$ are critical. The clique $K_a$ is safe once all its vertices are coloured or the following condition holds: 
\begin{itemize}
\item[(A)] $K_b$ is completely coloured and its colours form a subset of the colours used for $K_c$. 
\end{itemize}
Vertex $x_i$ is safe once it is coloured or conditions (A) and (B) hold. \nopagebreak
\begin{itemize}
\item[(B)] Vertex $y_i$ is coloured and has the same colour as $x_i$ or some vertex in $K_a, K_b$ or $K_c$. 
\end{itemize}


We give the following strategy. Note that the set of rules 1, 2, 5 and 6 below mirrors the bull strategy (Lemma~\ref{lemma:bull-strategy}) while rules 3 and 4 follow the ear graph strategy 
(Lemma~\ref{lemma:ear-strategy}).


\begin{enumerate}
\item If Bob colours $x_i$, Alice applies the same colour to $y_i$. 
\item If Bob colours $y_i$ with a colour from $K_a, K_b$ or $K_c$, Alice applies a new colour to $x_i$, while if Bob colours $y_i$ with a new colour, Alice applies the same colour to $x_i$.
\item If $K_b$ is not fully coloured and Bob colours a vertex in $K_b$ or $K_c$, Alice applies the same colour to a vertex in the other clique.
\item If $K_a$ has at least two uncoloured vertices and Bob colours one, Alice colours another.
\item Suppose $K_a$ has exactly one uncoloured vertex and Bob colours it with~$\alpha$. If vertices $x_i, y_i$ are coloured, Alice colours $x_{3-i}$ to win and if $x_i, y_i$ are uncoloured for both $i=1,2$, she first colours~$y_2$ with $\alpha$ and then $x_2$ with any feasible colour in the next round to win.
\item Suppose $K_b$ is fully coloured and Bob colours a vertex in $K_c$. Alice wins immediately if $x_i, y_i$ are coloured for some $i$. Otherwise she applies the same colour to $y_2$ and wins by making sure that $x_2$ is coloured in 
the next round.
\end{enumerate}

Observe that Alice's second move (colouring~$x_2$) in the last two rules is feasible, since at most $1+a+c$ colours have been used at that point. The first three rules ensure conditions (A) and (B) for the safety of $x_i$~and~$K_a$. Alice follows the first two rules at most twice and the third rule at most $|K_b|$ many times, while rules 5 and 6 immediately lead to winning moves. As long as Alice follows rules 1 to 4, vertices~$x_i, y_i$ are either both uncoloured or both coloured. 
Furthermore, in the latter case the vertex $y_i$ has the same colour as $x_i$ or some vertex in $K_a, K_b$ or $K_c$. Hence, as $K_a$ is safe when rule 5 or 6 is invoked and vertices $x_1$ and $x_2$ are already safe or made safe by the winning move, Alice wins.
\end{proof}

\begin{proposition}[$E_6$]\label{beweisEsechs}
Alice wins on $G\in E_6$ with $\max\{b+c+d,a+b\}$ colours.
\end{proposition}
\begin{proof}
	We have $\omega(G) = b+ \max\{ a, c+d \}$, hence $K_b$ and $K_d$ are critical. The vertex $e$ and clique $K_a$ may also be critical if they are completely connected. The following strategy again combines elements of the ear graph and the bull strategies (cf.~Lemmas~\ref{lemma:ear-strategy} and~\ref{lemma:bull-strategy}).

\begin{enumerate}
	\item If neither $K_a$ nor $K_c \vee K_d$ is fully coloured and Bob colours a vertex in one of these cliques, then Alice applies the same colour to the other clique.
	\item Suppose $K_a$ or $K_c \vee K_d$ is fully coloured. If Bob colours a vertex in the larger clique and fills it, Alice wins (as $K_d$ is fully coloured and $K_b$ is safe by assumption). Otherwise Alice colours another vertex in the larger clique (and wins by the same argument if she fills it).
	\item If $e$ is uncoloured and Bob colours a vertex in $K_b$, Alice applies the same colour to $e$.
	\item Suppose Bob colours $e$ with colour $\alpha$. Then rule 3 implies that $K_b$ is uncoloured and Alice can colour one of its vertices $v$. Hence if $\alpha$ is a colour that has not been used in $K_a$ or $K_c$, Alice applies $\alpha$ to $v$, otherwise she colours $v$ in a new colour.	
	\item Now suppose $e$ is coloured (with a colour from $K_b$ or $K_c$). If Bob colours a vertex in $K_b$ and fills it, Alice wins. Otherwise she colours another vertex in $K_b$ (and wins if she fills it).
\end{enumerate}

Observe that the strategy is correct because of the pairing $(K_b,e)$ and $(K_a,K_c\vee K_d)$ of the cliques of~$G$: for each pair, the colours of the smaller clique are a subset of those of the larger clique in the pair after each of Alice's move. This guarantees that there is a feasible colour for the chosen vertex at any given time in the game.
\end{proof}

\subsection{Strategies for Alice on the dragon structures ($E_4$ and $E_3$)}\label{subsec:stratB}
Here we present strategies for $E_3$ and $E_4$ that combine the two dragon strategies (cf.~Lemma~\ref{lemma:dragon-strategy}).

\begin{proposition}[$E_4$]\label{beweisEvier}
Alice wins on $G \in E_4$ with $\max\{m+1,n+1\}$ colours.
\end{proposition}
\begin{proof} Without loss of generality, we can assume $m \leq n$. Indeed, if vertex $b$ is missing and $m > n$, vertex $a$ is the only critical vertex and Alice wins by ensuring that it is coloured on her first move. If vertex $b$ is present, we can assume $m \leq n$ by symmetry. Note that every vertex may be critical at the beginning. Consider first the case $m=n=1$. Then $G$ without vertex $b$ is a dragon and Alice wins by Lemma~\ref{lemma:dragon-strategy}. If $b$ is present, then $G$ is an almost complete bipartite graph instance of $E_9$ and Alice wins. Hence we assume from now on that $n \geq 2$ and we have at least three colours at our disposal.
	
We give a strategy for Alice in two phases. In Phase 1, her basic strategy is to make sure that $K_m$ contains the same colours as $K_n$ until there is only one uncoloured vertex left in $K_n$. For Phase 2, we distinguish three different cases that represent different end games depending on the game state she finds herself in. Alice transitions from Phase 1 to Phase 2 if explicitly instructed in the rules below or once $K_n$ contains exactly one uncoloured vertex. Note that rule~1 in Phase 1 mirrors the pairing dragon strategy while rule 1 in Case 1 of Phase 2 is inspired by the bipartite dragon strategy. 

\paragraph{Phase 1:} As long as there are at least two uncoloured vertices in $K_n$ before Bob's move, Alice uses the following rule-based strategy.
\begin{enumerate}
\item If vertex $b$ does not exist and Bob colours vertex $a$, Alice applies the same colour to a vertex in $K_n$. Otherwise, if $b$ exists, and Bob colours a vertex in pair $(a,b)$ or $(c,d)$, Alice applies the same colour to the other vertex in the pair. In either case she proceeds to Case 3 of Phase 2.

\item If $K_n$ and $K_m$ both contain uncoloured vertices and Bob colours a vertex in one of them, Alice applies the same colour to a vertex in the other clique.

\item Suppose $K_m$ is fully coloured and Bob colours a vertex in $K_n$ with colour $\alpha$. Alice's response depends on the number of uncoloured vertices left in $K_n$. 
\begin{enumerate}[(i)]
\item If $K_n$ has at least three uncoloured vertices, Alice colours one of them. 
\item If $K_n$ has exactly two uncoloured vertices, Alice applies colour $\alpha$ to vertex $a$ and proceeds to Case 2 of Phase 2.
\item If $K_n$ has exactly one uncoloured vertex, Alice colours $a$ with $\alpha$. Now the remaining uncoloured vertices $(b,)c,d$ and the last vertex of $K_n$ induce a $C_4$ or $P_3$, depending on the presence of vertex $b$. Each of these vertices admits at least colours $n$ and $n+1$ and it is Bob's turn. Since $C_4$ is obviously $g_B$-perfect, Alice wins in the next round. 
\end{enumerate}
\end{enumerate}

\paragraph{Phase 2:}
\begin{enumerate}[{Case }1:]
\item In this case, vertices $a,b,c,d$ are still uncoloured and Alice entered Phase 2 because $K_n$ has exactly one uncoloured vertex. In fact, observe that we have $m \in \{n-1,n\}$, so $K_m$ has at most one uncoloured vertex. We assume without loss of generality that the colours of $K_m$ and $K_n$ are $1,\ldots,n-1$. Alice proceeds by the following rules.

\begin{enumerate}[1.]
\item If Bob colours $a$ or $b$ with colour $n$, Alice colours $c$ with colour $n+1$. Similarly, if Bob colours $c$ or $d$ with colour $n$, Alice colours $a$ with colour $n+1$ and wins.
\item If Bob colours the last uncoloured vertex of $K_n$ or $K_m$ with colour $n$, Alice wins by colouring the last vertex of the other clique with colour $n+1$. If $K_m$ is already fully coloured, she instead colours $c$ with colour $n+1$ to win.
\end{enumerate}

\item In this case, Alice transitioned into Phase 2 from rule 2 in Phase 1. Hence $K_m$ is fully coloured with a subset of colours $1,\ldots,n-3$, the clique $K_n$ is coloured with colours $\{1,\ldots,n-2\}$ and has exactly two uncoloured  vertices, vertex $a$ is coloured with colour $n-2$, and $b$ (if present), $c$~and $d$ are uncoloured. Alice wins in a single round, as follows.
\begin{enumerate}[1.]
	\item If Bob colours $c$ or $d$, Alice wins by applying the same colour to other vertex.
	\item If Bob colours $b$, Alice wins by colouring an uncoloured vertex of $K_n$ such that $b$ has the same colour as a vertex in $K_n$.
	\item If Bob colours an uncoloured vertex of $K_n$, Alice colours $b$ with colour $n-2$ if $b$ is present, otherwise the last uncoloured vertex of $K_n$ with any feasible colour. In any case, she wins.
\end{enumerate}

\item In this case, Alice entered Phase 2 because Bob coloured vertex $a,b,c$ or $d$. If $b$ is present, then either pair $(a,b)$ or pair $(c,d)$ is coloured and $K_n$ has at least two uncoloured vertices. If $b$ is not present, then  either $(c,d)$ is coloured, $a$ is uncoloured and $K_n$ has at least two uncoloured vertices, or $a$ is coloured with a colour from $K_n$, $(c,d)$ is uncoloured and $K_n$ has at least one uncoloured vertex. This guarantees the feasibility of Alice's following moves.

Note that if $b$ does not exist and pair $(c,d)$ is coloured, then only vertex $a$ is critical and Alice wins by making sure that $a$ is coloured on her next move. If $b$ does exist, $(c,d)$ is coloured and $m < n$, then all vertices are safe and Alice wins immediately. Otherwise, if $m = n$, we can relabel $(a,b)$ to $(c,d)$ and $K_m$ to $K_n$ (and vice versa). Hence we assume without loss of generality that $a$ is coloured and $(c,d)$ is not. If $b$ exists, then it has the same colour as $a$, and if~$b$ does not exist, then some vertex in $K_n$ has the same colour as $a$. Alice proceeds by the following rules. If Bob colours $c$ or $d$, Alice wins by colouring the other vertex the same. If Bob instead colours a vertex in $K_m$ or $K_n$, Alice colours a vertex in $K_n$. The first time she does this, she makes $c$ and $d$ safe by ensuring that $K_n$ contains the colour of $a$ (and $b$). Once $K_n$ is fully coloured, she wins.
\end{enumerate}
\end{proof}

\begin{proposition}[$E_3$]\label{beweisEdrei}
Alice wins on $G \in E_3$ with $\omega(G)$ colours.
\end{proposition}

\begin{proof}
First, let $G$ be a graph in $E_3$ without vertex $c$. Without loss of generality, assume $m \geq 1, n \geq 2$, as $G \in E_4$ for $n=1$. The clique number of $G$ is $\omega(G) = m+n \geq m+2$. The cliques $K_m$ and $K_n$ and possibly vertex $a$ are critical. Note that every critical vertex can be saved by colouring two neighbours with the same colour. Thus in her strategy, Alice simply makes sure that $a$ is coloured the same as a vertex in $K_m$ and $b$ is coloured the same as a vertex in $K_n$. This procedure follows the pairing dragon strategy.

\begin{enumerate}
\item Suppose $a$ is uncoloured. If Bob colours $a$ or a vertex in $K_m$, Alice applies the same colour to the other option. This saves $b$ and $K_n$. 
\item Suppose $a$ is coloured. If Bob colours a vertex in $K_m$, Alice colours another vertex in $K_m$. If $K_m$ is completely coloured at that point or before her move, then every vertex is safe and Alice wins by Observation~\ref{observation:winning}.
\item If Bob colours $b$ or $K_n$, Alice follows rules 1 and 2 with $(a,K_m)$ substituted by $(b,K_n)$.
\item Lastly, if Bob colours $d$, Alice colours a vertex in $K_n$. If possible, she uses the same colour as Bob, otherwise a new colour. If $K_n$ is fully coloured before or after her move, then every vertex is safe and Alice wins by Observation~\ref{observation:winning}.
\end{enumerate}


Now let $G$ be an instance of $E_3$ with vertex $c$. By definition, $m,n \geq 1$ and the only non-critical vertex is~$d$. For a colouring with~$\omega(G) = m+n+1$ colours, 
\begin{itemize}
\item[(1)] vertex~$b$ must be coloured the same as a vertex in~$K_n$, 
\item[(2)] vertex~$a$ must be coloured the same as~$c$ or a vertex in~$K_m$, 
and 
\item[(3)] vertex~$d$ must be coloured the same as a vertex in~$K_m$ or~$K_n$. 
\end{itemize}
Therefore, we propose the following strategy for Alice, which is inspired by the proper colouring of the base graph of $G$ shown Figure~\ref{bildbasegraphfourwheel}.

\begin{figure}
\begin{center}
\includegraphics[scale=0.5]{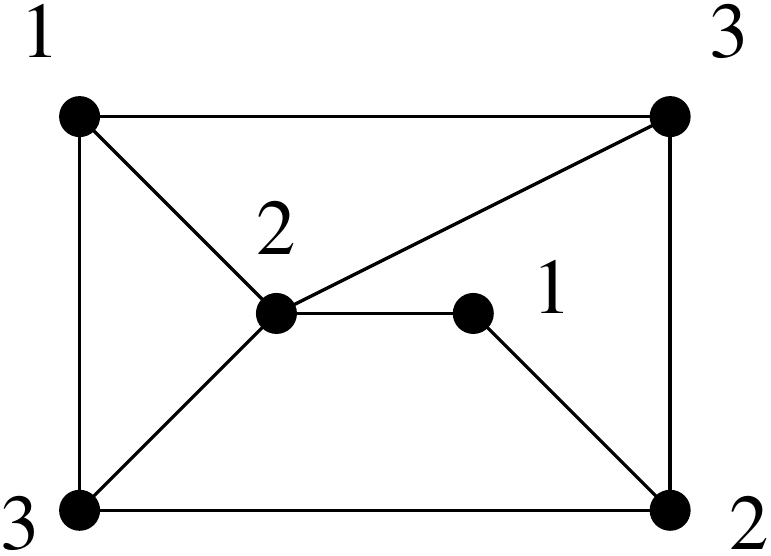}
\end{center}
\caption{\label{bildbasegraphfourwheel}A proper colouring of a 4-wheel-subdivision, the base graph of an $E_3$ instance with vertex $c$.}
\end{figure}

\begin{enumerate}
\item If Bob colours $a$ with a colour from $K_m$, Alice colours $c$ with a new colour. Otherwise, if Bob colours $a$ or $c$, Alice applies the same colour to the other vertex.
\item Suppose $K_n$ is completely uncoloured and Bob colours $b$ or a vertex in $K_n$. Then Alice applies the same colour to the other option.
\item Suppose Bob colours a vertex in $K_n$ that already has coloured vertices. If $K_n$ and $d$ are now fully coloured, Alice wins. Otherwise, if $K_n$ is fully coloured and $d$ is uncoloured, she colours $d$ with the colour Bob used last and wins. Finally, if $K_n$ has uncoloured vertices, Alice colours one of these.
\item Suppose $K_m$ is completely uncoloured and Bob colours $d$ or a vertex in $K_m$. Then Alice applies the same colour to the other option.
\item Suppose Bob colours a vertex in $K_m$ that already has coloured vertices. If $K_m$ and $a$ are now fully coloured, Alice wins. Otherwise, if $K_m$ is fully coloured and $a$ is uncoloured, she colours $a$ with the same colour and wins. Finally, if $K_m$ has uncoloured vertices, Alice colours one of these.
\end{enumerate}

This strategy guarantees that every move is feasible and conditions (1), (2) and (3) are satisfied at the end. Hence Alice wins by Observation~\ref{observation:winning}.
\end{proof}

\section{Complexity Results}\label{section:complexity-results}
\subsection{The clique module decomposition}\label{subsec:clique-module-decomposition}
Let $G$ be a graph with $n$ vertices. Without loss of generality, we assume that the vertices are numbered from 1 to $n$. Throughout this section we use the adjacency matrix of the graph, which can be constructed in $O(n^2)$ time and enables us to check adjacency of two vertices in $O(1)$ time. 

A subset $S \subseteq V(G)$ is called a \emph{module} in $G$ if every $v \in V(G) \setminus S$ is either adjacent to all or none of $S$. If $S$ is also a clique, we call it a \emph{clique module}. A clique module $C$ is \emph{maximal} if no other clique module $C'$ with $C \subset C'$ exists. Finally, a \emph{clique module decomposition} $\mathcal{F}_G$ of $G$ is a partition of $V(G)$ into maximal clique modules. The \emph{base graph} $B_G$ of $G$ is obtained by contracting each maximal clique module in $\mathcal{F}_G$ to a single vertex while respecting the adjacencies of the original graph $G$. Theorem~\ref{thm:clique-module-decomposition} shows that this notion is well-defined. We call a vertex in the base graph a \emph{clique vertex} if it is obtained by contracting two or more vertices, and a \emph{singleton vertex} otherwise.

Algorithm~\ref{algo:clique-module-decomposition} computes a clique module decomposition $\mathcal{F}_G$ of $G$ in $O(n^2)$ time. In order to construct the base graph $B_G$, we identify its vertices with the sets in the clique module decomposition $\mathcal{F}_G$. Next we query for every pair $C, T \in \mathcal{F}_G$ whether the first vertex in $C$ and $T$ are adjacent and add the edge $(C,T)$ in $B_G$ if so. As there are $|\mathcal{F}_G|^2$ pairs and 
$|\mathcal{F}_G| \leq n$, this takes $O(n^2)$ time in total. Recall that $N[v] := N(v) \cup \{v\}$ denotes the set of neighbours of $v$ together with $v$ itself.

\begin{algorithm}
	\caption{A simple clique module decomposition algorithm.}
	\label{algo:clique-module-decomposition}
	\begin{algorithmic}
		\State Set $\mathcal{F} = \{ \{1, \ldots, n \} \}$
		\For {$v=1, \ldots, n$}
			\ForAll {$S$ in $\mathcal{F}$}
				\State Remove $S$ and add $S \cap N[v]$ and $S \setminus N[v]$ to $\mathcal{F}$ unless the respective set is empty.
			\EndFor
		\EndFor
		\State Return $\mathcal{F}$.
	\end{algorithmic}
\end{algorithm}


\begin{lemma}\label{lemma:clique-module-algo-correctness}
	Given a graph $G$, Algorithm~\ref{algo:clique-module-decomposition} returns a clique module decomposition in $O(n^2)$ time.
\end{lemma}
\begin{proof}
	Note that $\mathcal{F}$ is a partition of $V(G)$ throughout the running time of the algorithm. Let $\mathcal{F}^*$ be the final partition that is returned on running the algorithm on $G$. First we show that every $C \in \mathcal{F}^*$ is a clique module. Note that after $k$ executions of the outer loop we know that for any vertex $1 \leq v \leq k$, all vertices in the unique set $C$ containing $v$ are adjacent to $v$. In particular, for $k = n$ this implies that every $C \in \mathcal{F}^*$ is a clique. Now fix $C \in \mathcal{F}^*$. Suppose $C$ is not a module, i.e.~there exist vertices $v,w \in C$ and $z \in V(G) \setminus C$ with $vz \in E(G)$ and $wz \not \in E(G)$. But this is impossible, as $v$ and $w$ would have been separated on the $z$-th execution of the outer loop and can no longer be in the same set. Finally, we note that the clique modules in $\mathcal{F}^*$ are maximal. For this, let $D$ be a maximal clique module in $G$. Every time a superset $S \supseteq D$ is replaced by $S \cap N[v]$ and $S \setminus N[v]$, $D$ is either a subset of the former or the latter. This implies $D \subseteq C$ for some $C \in \mathcal{F}^*$.
	
	Clearly, the outer loop runs $n$ times. On the other hand, determining $S \cap N[v]$ and $S \setminus N[v]$ can be done in $O(|S|)$ time for each $S \in \mathcal{F}$. Hence for each execution of the outer loop, the inner loop takes $\sum_{S \in \mathcal{F}}O(|S|) = O(n)$ time. In total, we get a running time of $O(n^2)$. 
\end{proof}

\begin{lemma}\label{lemma:clique-module-union}
	If $C$ and $D$ are clique modules in $G$ that share a vertex $v$, then their union $C \cup D$ is also a clique module.
\end{lemma}
\begin{proof}
	To see that $C \cup D$ is a clique, first note that $C$ and $D$ are cliques. Then, for any $x \in C$ and $y \in D$, the existence of edge $xy$ is implied by $vy \in E$, as $x$ and $v$ are in the same clique module $C$. Now let $z \in V \setminus (C \cup D)$. Then $z$ is adjacent to $v$ if and only if it is adjacent to all vertices of $C$ and $D$, as the two clique modules share the vertex $v$.
\end{proof}

\begin{theorem}\label{thm:clique-module-decomposition}
	Every graph $G$ has a unique clique module decomposition.
\end{theorem}
\begin{proof}
	By Lemma~\ref{lemma:clique-module-algo-correctness}, we know that every graph $G$ has a clique module decomposition. Now suppose $\mathcal{F}$ and $\mathcal{F}'$ are two such decompositions. Fix a vertex $v \in V$ and 
$C \in \mathcal{F}$ and $D \in \mathcal{F}'$ with $v \in C \cap D$. By Lemma~\ref{lemma:clique-module-union}, we know that $C \cup D$ is also a clique module. This implies that, in order for $C$ and $D$ to both be maximal clique modules, 
we must have $C = D$. It follows that $\mathcal{F} = \mathcal{F'}$.
\end{proof}

\subsection{Complexity results}\label{subsec:complexity-results}
Deciding whether a given graph $G$ is $g_B$-perfect is in P. This follows immediately from our forbidden subgraph characterisation, which implies the $\Theta(n^7)$-time Algorithm~\ref{algo:naive-perfectness}. We can significantly improve on 
this 
by utilising our explicit structural characterisation. Exploiting the clique module decomposition technique, we can determine in quadratic time whether $G$ is in $E_i$ for any $e \in \{1, \ldots, 15 \}$. This implies our complexity results, restated below for convenience.

	\begin{algorithm}
		\caption{A naive algorithm for checking $g_B$-perfectness}
		\label{algo:naive-perfectness}
		\begin{algorithmic}
			\ForAll {$5$-subsets $S$ of $V(G)$}
				\State Return \texttt{false} if the subgraph of $G$ induced by $S$ matches one of $F_1, \ldots, F_8$.
			\EndFor
			\ForAll {$7$-subsets $S$ of $V(G)$}
				\State Return \texttt{false} if the subgraph of $G$ induced by $S$ matches one of $F_9, \ldots, F_{15}$.
			\EndFor
			\State Return \texttt{true}.
		\end{algorithmic}
	\end{algorithm}

\medskip
\noindent \textbf{Theorem~\ref{thm:recognition}.} \emph{There is an $O(n^2)$ time algorithm deciding whether a graph $G$ with $n$ vertices is $g_B$-perfect (or $g_A$-perfect).}

\begin{proof}
We show that deciding whether a graph $G$ with $n$ vertices is an instance of one of the graph classes $E_1^\cup, E_2, \ldots, E_9$ takes time $O(n^2)$. Hence, 
by Theorem~\ref{thm:gB-perfect-characterisation},
we can recognise $g_B$-perfect graphs in $O(n^2)$ time.
Consider the following simple subroutines to determine membership in each graph class $E_1^\cup,E_2,\ldots,E_9$. 
	
\begin{description}
	\item[Subroutine for $E_1$.] Compute the base graph $B_G$ of $G$ and store it as $H$. If $H$ has a dominating vertex~$x$, then remove it, else return \texttt{false}. Remove all isolated vertices in $H$. If $H$ is null or a $P_3$, then 
return \texttt{true}, else \texttt{false}.

	\item[Subroutine for $E_1^\cup$.] Compute the components of~$G$ and run the above subroutine for $E_1$ on each. Return \texttt{true} if and only if each component is in $E_1$.

	\item[Subroutines for $E_2, E_3, E_4$ and $E_6$.] Compute $B_G$ and its order $n$. If $n \not = |B_{E_i}|$, then return \texttt{false}. Return \texttt{true} if one of the permutations $\pi$ of $V(B_G)$ is an isomorphism from $B_G$ to $B_{E_i}$ mapping every clique vertex in $B_G$ to a clique vertex in $B_{E_i}$.

	\item[Subroutine for $E_5$.] Compute the complement $\overline{B_G}$ of $B_G$ and the bipartition $(A,Z)$ of $\overline{B_G}$ or return \texttt{false} if the graph is not bipartite. Remove any vertex in $A$ or $Z$ that is adjacent to every vertex in the other set. Return \texttt{true} if $|A| = |Z|$ and the graph is $(|A|-1)$-regular, otherwise return \texttt{false}.
	
	\item[Subroutine for $E_7$.] Compute the number $m$ of edges in $B_G$ and the bipartition $(U,V)$ of $B_G$ or return \texttt{false} if the graph is not bipartite. Return \texttt{false} if $|U| \not = 2 \not = |V|$, else without loss of generality we have $|U| = 2$. Return \texttt{true} if $m \in \{ 2|V|-1, 2|V| \}$, $U$ has no clique vertices and $V$ has at most one clique vertex, otherwise return \texttt{false}.
	
	\item[Subroutine for $E_8$.] Compute the number $m$ of edges in $B_G$ and the bipartition $(U,V)$ of $B_G$ or return \texttt{false} if the graph is not bipartite. Return \texttt{false} if $|U| \not = 3 \not = |V|$, else without loss of generality we have $|U| = 3$. Return \texttt{true} if $m  = 3|V|$, $U$ has no clique vertices and $V$ has at most one clique vertex, otherwise return \texttt{false}.
	
	\item[Subroutine for $E_9$.] Compute the bipartition $(U,V)$ of $G$ (not $B_G$) or return \texttt{false} if $G$ is not bipartite. If $G$ has $|U||V|-1$ or $|U||V|$ edges, then return \texttt{true}, else \texttt{false}.
\end{description}

By Lemma~\ref{lemma:clique-module-algo-correctness}, computing the base graph takes $O(n^2)$ time. In the subroutines for $E_2, E_3, E_4$ and $E_6$, note that $|B_{E_i}| \leq 7$ for $i \in \{2,3,4,6 \}$, so that checking for isomorphisms 
takes constant time. Finally, removing dominating or isolated vertices as well as computing the complement and the bipartition of a graph also uses quadratic time, so each subroutine is quadratic.
As remarked above this proves that we can recognise $g_B$-perfect graphs in $O(n^2)$ time.
%

The following subroutine tests whether a given graph $G$ is $g_A$-perfect and runs in quadratic time by the same arguments as in the proof above.
\begin{description}
	\item[Subroutine for $g_A$-perfect graphs.] Compute the components of $G$ and run the subroutine for $E_1$ on each of them. Return \texttt{true} if all components are in $E_1$ and \texttt{false} if 
at least two components are not in~$E_1$. 
Suppose exactly one component $C_1$ is not in $E_1$. If $C_1$ has no dominating vertex, then return \texttt{false}. Else run the subroutine for $E_1$ on each component of $C_1 - x$, where $x$ is a dominating vertex of $C_1$, and return \texttt{true} if and only if 
each component is in~$E_1$.\nopagebreak 
\end{description}\nopagebreak
This proves that we can also recognise $g_A$-perfect graphs in $O(n^2)$ time.
\end{proof}


We now discuss some consequences of Theorem~\ref{thm:recognition}. First we consider Corollary~\ref{corollary:alice-winning-time}, restated below for convenience.

\medskip
\noindent \textbf{Corollary~\ref{corollary:alice-winning-time}.} \emph{Alice can win on any $g_A$- or $g_B$-perfect graph $G$ with $\omega(G)$ colours using only $O(n^2)$ computational time.}

\medskip

\begin{proof}
By Theorem~\ref{thm:recognition},
Alice can check in quadratic time whether the graph $G$ is game-perfect. If it is not, Alice is guaranteed to lose. If it is, Alice can identify, using the 
subroutines from the proof of Theorem~\ref{thm:recognition}, to which class 
$E_1^\cup,E_2,\ldots,E_9$ the graph belongs and uses the strategy for this class from Section~\ref{section:strategies} for game~$g_B$ and the instructions from the simple strategy given by \citet{andres1} for game~$g_A$. It is easy to check 
that following instructions from each strategy requires at most quadratic time. Thus Alice can win on any $g_A$- or $g_B$-perfect graph in quadratic time.
\end{proof}

Next we turn to Corollary~\ref{corollary:hamilton}, restated for convenience.

\medskip
\noindent \textbf{Corollary~\ref{corollary:hamilton}.} \emph{\textsc{Hamilton Cycle} is in P for $g_A$- and $g_B$-perfect graphs.}

\medskip

A result by \citet{babeletal} shows that \textsc{Hamilton Cycle} on graphs with ``few'' $P_4$s takes linear time ($O(n+m)$ for a such graphs with $n$ vertices and $m$ edges). As this class of graphs contains the $g_A$-perfect graphs, 
Corollary~\ref{corollary:hamilton} holds for the game~$g_A$.

\begin{theorem}[\citet{babeletal}]\label{theorem:babel}
	 A \emph{$(q,q-4)$-graph} is a graph such that every set of at most $q$ vertices contains at most $q-4$ distinct induced $P_4$s. For every integer $q \geq 4$ there exists a linear time algorithm that decides whether a $(q,q-4)$-graph is Hamiltonian.
\end{theorem} 

\begin{corollary}
\textsc{Hamilton Cycle} is in P for $g_A$-perfect graphs.
\end{corollary}
\begin{proof}
By Theorem~\ref{thm:gA-gAB-perfect-characterisation}, $g_A$-perfect graphs are $P_4$-free and thus $(4,0)$-graphs.
\end{proof}

When considering the game $g_B$, we note that graphs in $E_1^\cup$, $E_2$ and $E_7$ are $(q,q-4)$-graphs but this approach breaks down for the remaining classes. Instead, our proof argues about the structure of graphs using the clique module decomposition and the following three observations.

\begin{observation}\label{observation:sufficient-hamiltonian}
	A graph is Hamiltonian if its base graph is (but the converse need not be true).
\end{observation}

\begin{observation}\label{observation:cut-condition}
	A graph with a cut vertex is non-Hamiltonian.
\end{observation}

\begin{observation}\label{observation:bipartite-hamiltonian}
	A non-empty complete bipartite graph $K_{n,m}$ is Hamiltonian if and only if $m = n \geq 2$. An almost complete bipartite graph $K_{m,n}-e$ is Hamiltonian if and only if $m=n \geq 3$.
\end{observation}

\begin{proof}[of Corollary~\ref{corollary:hamilton}]
Let $G=(V,E)$ be a $g_B$-perfect graph. By Theorem~\ref{thm:recognition}, we can determine in quadratic time which class(es) $G$ belongs to. We give Hamiltonicity criteria for each graph class that can be efficiently checked, which implies our result. 

Graphs in $E_1$ are Hamiltonian only if $a,b,c \geq 1$ and $k=0$ or $a=b=c=0$, $k=1$ and $|V| \geq 3$. Hence the graph is Hamiltonian if and only if $|V| \geq 3$ and its base graph is either a single clique vertex or a $P_3$ with a clique vertex as its middle vertex.

Observations~\ref{observation:cut-condition} and~\ref{observation:sufficient-hamiltonian} imply that graphs in $E_2$ are non-Hamiltonian, and graphs in $E_3$ and $E_4$ are Hamiltonian if and only if the optional vertex is present.

As the base graph of graphs in $E_5$ with $k \geq 2$ is Hamiltonian, the graph itself is also Hamiltonian. Additionally, if $k=1$, the graph is Hamiltonian if and only if $A_1$ and $Z_1$ both contain at least two vertices, or $A_R$ ($Z_R$) is empty and $Z_1$ ($A_1$) contains at least two vertices.

Extended house graphs ($E_6$) are Hamiltonian, by Observation~\ref{observation:sufficient-hamiltonian}. 
Extended bull graphs ($E_6$) are \mbox{Hamiltonian} if and only if $b, d \geq 2$, by Observation~\ref{observation:cut-condition}. 
The same observation also implies that a graph in $E_7$ is Hamiltonian only if $c$ is connected to both $K_n$ 
and $b$. Hence by Observation~\ref{observation:bipartite-hamiltonian}, such a graph is Hamiltonian if and only if its base graph is a $K_{2,2}$. Similarly, a graph in $E_8$ is Hamiltonian if and only if its base graph is a $K_{3,3}$ 
and  a graph in $E_9$ is Hamiltonian if it satisfies the condition in Observation~\ref{observation:bipartite-hamiltonian} together with the observation that the null graph, which is a member of~$E_9$, is Hamiltonian. 
\end{proof}

\section{Further work}\label{section:further-work}
Following the characterisation of $g_B$-perfect graphs by means of forbidden induced subgraphs and explicit structural descriptions, we ask whether such characterisations can be obtained for the remaining uncharacterised games $g_{B,A}$ and $g_{A,A}$.

\begin{problem}\label{problem:BA-and-AA-characterisation}
	Characterise the $g_{B,A}$ and $g_{A,A}$-perfect graphs by a set of forbidden induced subgraphs and/or explicit structural descriptions.
\end{problem}

\begin{figure}[htb]
\begin{center}
\includegraphics[scale=0.36]{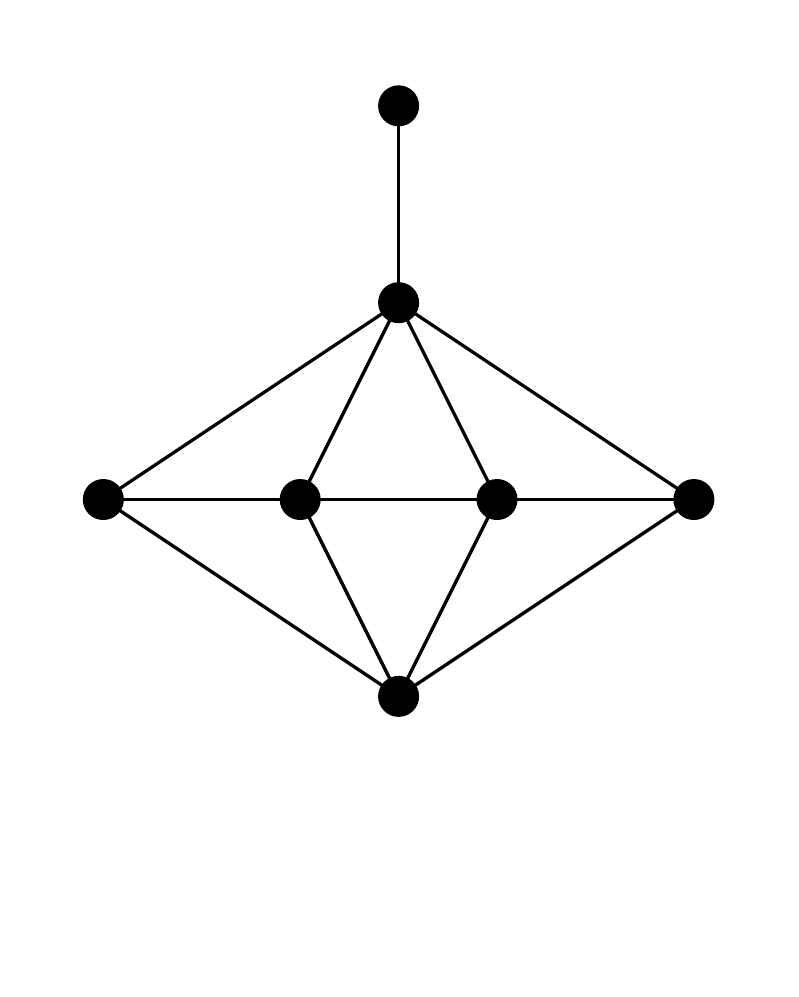}
\includegraphics[scale=0.36]{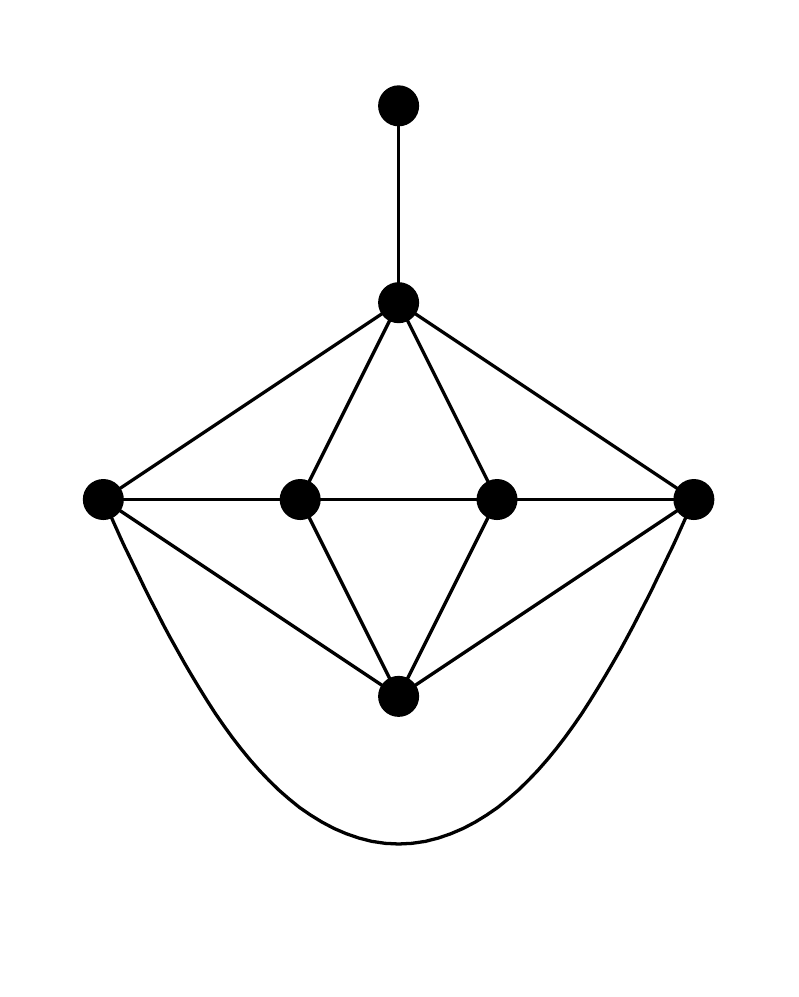}
\includegraphics[scale=0.36]{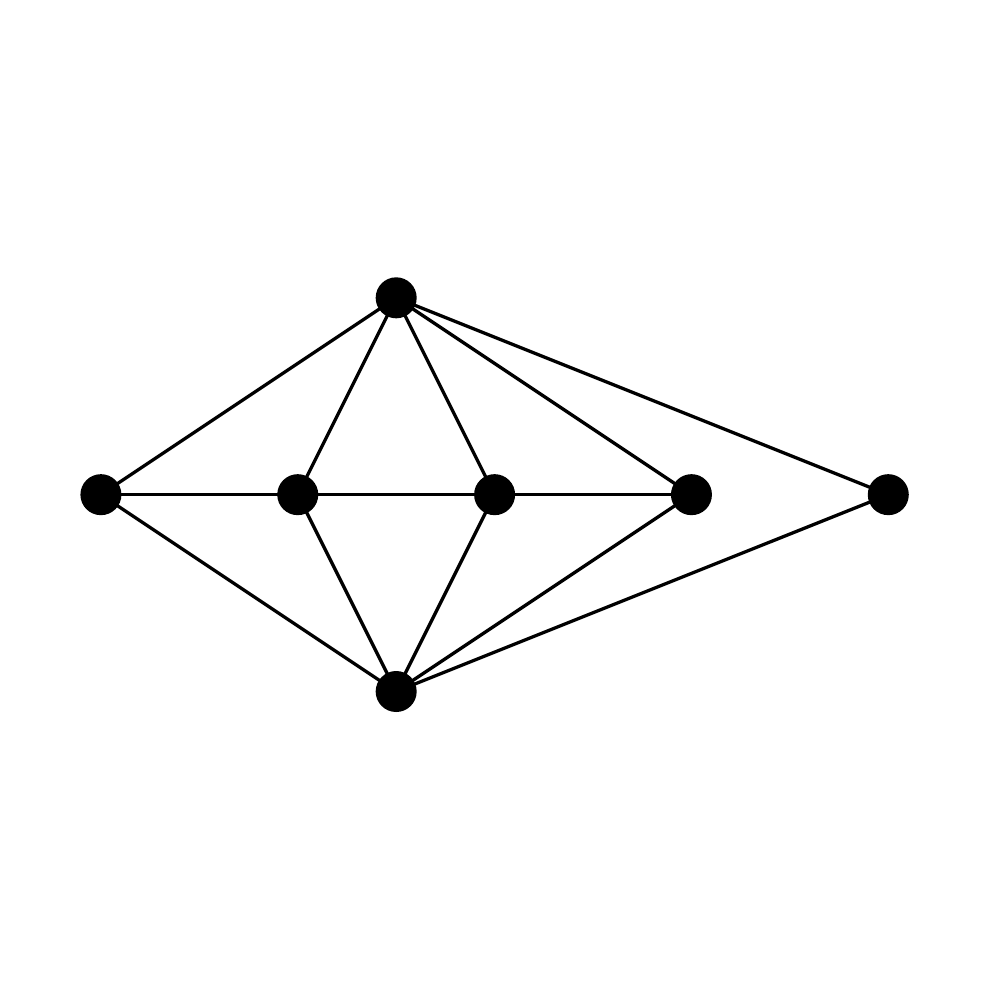}
\includegraphics[scale=0.36]{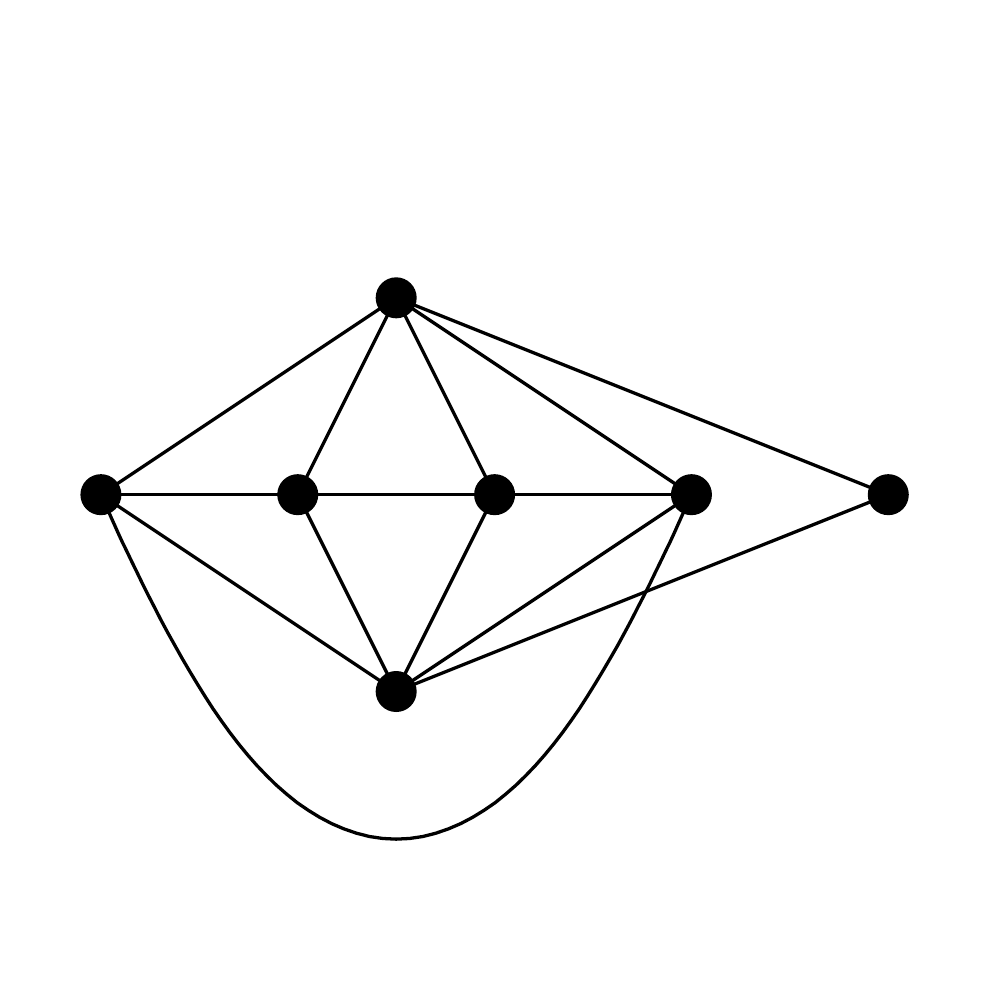}

\includegraphics[scale=0.36]{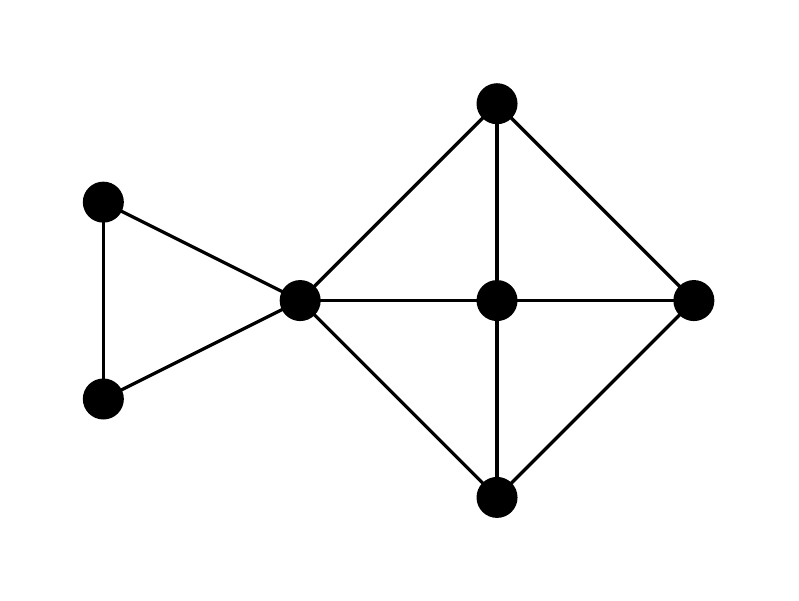}
\includegraphics[scale=0.36]{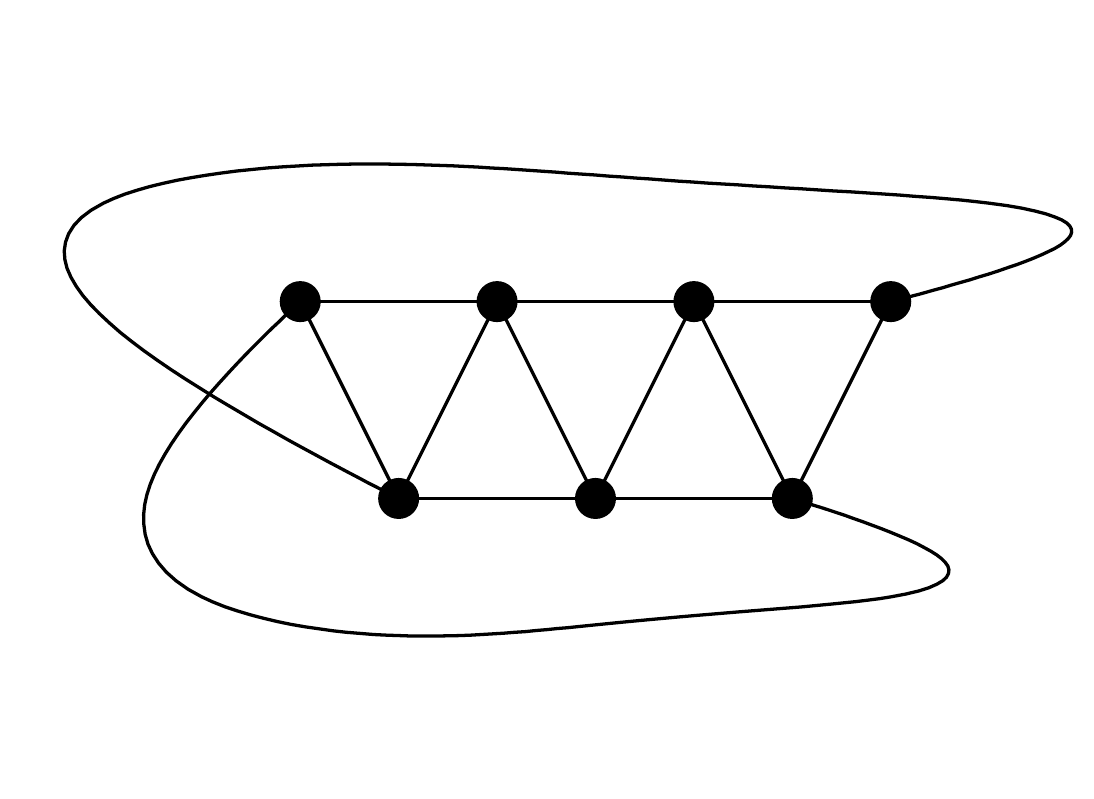}
\includegraphics[scale=0.36]{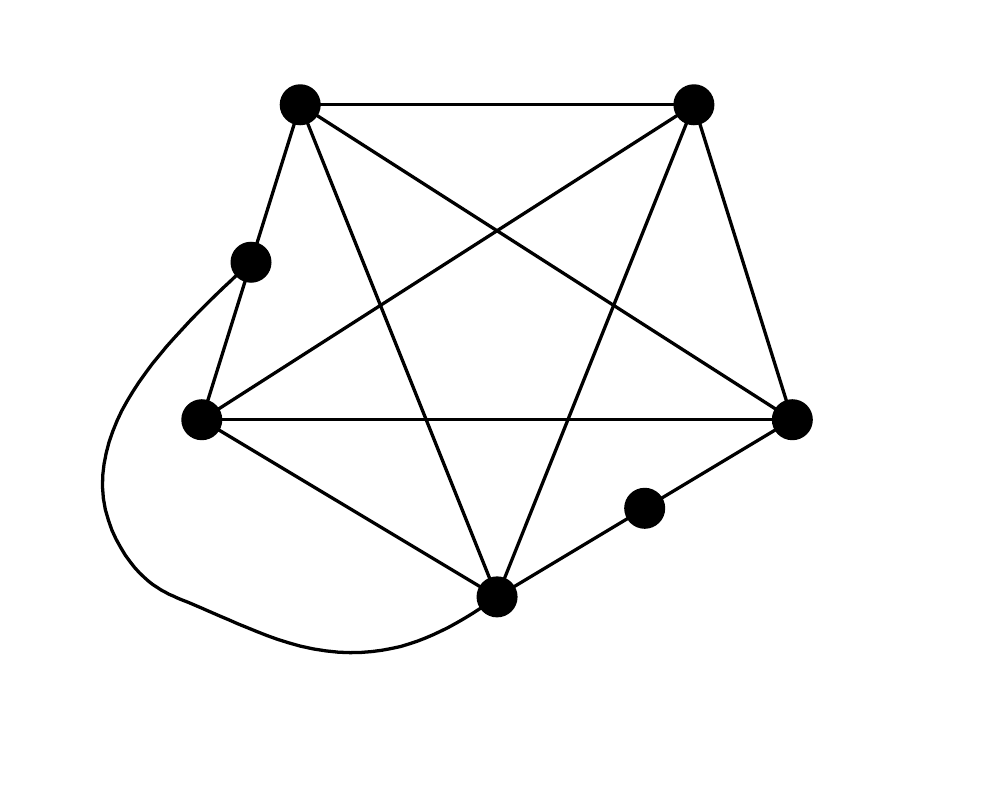}
\end{center}
\caption{7 of the 16 minimal forbidden configurations with up to 10 vertices for $g_{B,A}$-perfect graphs.}
\label{knownBA}
\end{figure}

 Partial progress on this has been made. \citet{lockbachelor} performed an exhaustive computer search to determine all minimal forbidden configurations with at most 10 vertices for the games $g_{B,A}$ and $g_{A,A}$. For the game 
$g_{B,A}$, this has yielded the minimal induced forbidden subgraphs $F_1$, $F_2$, $F_5$, $F_6$, $F_9$, $F_{10}$, $F_{11}$ (see Figure~\ref{fig:gB-forbidden-graphs}), the seven graphs depicted in Figure~\ref{knownBA}, and the odd antiholes $\overline{C_7}$ and $\overline{C_9}$. For the game $g_{A,A}$, 73 minimal forbidden induced subgraphs were found. In addition, we note that the disjoint union of two double fans is minimally forbidden for $g_{A,A}$, demonstrating that minimal forbidden subgraphs with more than 10 vertices exist for this game. These results suggests that the classes of $g_{A,A}$-perfect and $g_{B,A}$-perfect graphs are substantially richer than the $g_B$-perfect graphs and new ideas are required to characterise these.

The following known results imply that the set of minimal induced forbidden subgraphs is infinite. In particular, Corollary~\ref{cor:antiholes} states that all odd antiholes ($\overline{C_k}$ with odd $k \geq 5$) are minimal 
 forbidden induced subgraphs for the games $g_{B,A}$ and $g_{A,A}$. Note that in the context of game $g_B$, the odd antiholes of order $k\ge7$ are also forbidden induced subgraphs but are not minimal forbidden since they contain a 
forbidden induced 4-fan ($F_4$ in Figure~\ref{fig:gB-forbidden-graphs}).

\begin{theorem}[\citet{andres4}]\label{theorem:AA-complements-bipartite}
Complements of bipartite graphs are $g_{A,A}$-perfect.
\end{theorem}
\begin{theorem}\label{theorem:BA-complements-bipartite}
Complements of bipartite graphs are $g_{B,A}$-perfect.
\end{theorem}
\begin{proof}
	Identical to the proof of Theorem~\ref{theorem:AA-complements-bipartite}.
\end{proof}


\begin{cor}\label{cor:antiholes}
Odd antiholes are minimal forbidden configurations in $g_{B,A}$- and $g_{A,A}$-perfect graphs.
\end{cor}

\begin{proof}
Odd antiholes are forbidden, since they are not even perfect. Every proper induced subgraph of an odd antihole is a complement of a forest of paths and thus the complement of a bipartite graph. Hence we can apply Theorem~\ref{theorem:AA-complements-bipartite} and Theorem~\ref{theorem:BA-complements-bipartite}, respectively.
\end{proof}

These results motivate us to conjecture the following.
\begin{conj}\label{conjBA}
	A graph is $g_{B,A}$-perfect if and only if it contains no odd antihole of order $k\geq 7$ or any of the 14 configurations listed above as an induced subgraph.
\end{conj}

Recall that \citet{strongperfectgraphtheorem} characterised the class of perfect graphs as the graphs without induced odd holes and odd antiholes, while a characterisation in terms of explicit structural descriptions remains elusive. Such a result might be of major algorithmic interest in computer science, and could also lead to a new proof of the Strong Perfect Graph Theorem using a triple equivalence as in the formulation of Theorem~\ref{thm:gB-perfect-characterisation}. Since characterising game-perfectness for the games $g_{B,A}$ and $g_{A,A}$ involves odd antiholes by Corollary~\ref{cor:antiholes}, methods developed towards a solution of Problem~\ref{problem:BA-and-AA-characterisation} might provide some insights into Problem~\ref{prob:perfect}.

\begin{prob}\label{prob:perfect}
	Characterise the class of perfect graphs by means of explicit structural descriptions.
\end{prob}

\acknowledgements
\label{sec:ack}
We thank Vaidy Sivaraman for valuable comments and an anonymous reviewer for their helpful suggestions on the presentation of the paper.

\bibliographystyle{abbrvnat}
\bibliography{literature}

\begin{thebibliography}{32}
\providecommand{\natexlab}[1]{#1}
\providecommand{\url}[1]{\texttt{#1}}
\expandafter\ifx\csname urlstyle\endcsname\relax
  \providecommand{\doi}[1]{doi: #1}\else
  \providecommand{\doi}{doi: \begingroup \urlstyle{rm}\Url}\fi

\bibitem[Andres(2009)]{andres4}
S.~D. Andres.
\newblock Game-perfect graphs.
\newblock \emph{Math. Methods Oper. Res.}, 69\penalty0 (2):\penalty0 235--250,
  2009.

\bibitem[Andres(2012)]{andres1}
S.~D. Andres.
\newblock On characterizing game-perfect graphs by forbidden induced subgraphs.
\newblock \emph{Contrib. Discrete Math.}, 7\penalty0 (1):\penalty0 21--34,
  2012.

\bibitem[Babel et~al.(2001)Babel, Kloks, Kratochv{\'{i}}l, Kratsch,
  M{\"{u}}ller, and Olariu]{babeletal}
L.~Babel, T.~Kloks, J.~Kratochv{\'{i}}l, D.~Kratsch, H.~M{\"{u}}ller, and
  S.~Olariu.
\newblock Efficient algorithms for graphs with few {$P_4$}'s.
\newblock \emph{Discrete Math.}, 235\penalty0 (1-3):\penalty0 29--51, 2001.

\bibitem[Bodlaender(1991)]{bodlaender}
H.~L. Bodlaender.
\newblock On the complexity of some coloring games.
\newblock \emph{Internat. J. Found. Comput. Sci.}, 2\penalty0 (2):\penalty0
  133--147, 1991.

\bibitem[Bohman et~al.(2008)Bohman, Frieze, and Sudakov]{bohman}
T.~Bohman, A.~Frieze, and B.~Sudakov.
\newblock The game chromatic number of random graphs.
\newblock \emph{Random Structures Algorithms}, 32\penalty0 (2):\penalty0
  223--235, 2008.

\bibitem[Cai and Zhu(2001)]{caizhu}
L.~Cai and X.~Zhu.
\newblock Game chromatic index of $k$-degenerate graphs.
\newblock \emph{Journal of Graph Theory}, 36\penalty0 (3):\penalty0 144--155,
  2001.

\bibitem[Charpentier and Sopena(2013)]{charpentiersopena}
C.~Charpentier and {\'{E}}.~Sopena.
\newblock Incidence coloring game and arboricity of graphs.
\newblock In T.~Lecroq and L.~Mouchard, editors, \emph{Combinatorial
  Algorithms. IWOCA 2013. Lecture Notes in Computer Science, vol. 8288}, pages
  106--114, Berlin, Heidelberg, 2013. Springer.

\bibitem[Chudnovsky et~al.(2005)Chudnovsky, Cornu{\'{e}}jols, Liu, Seymour, and
  Vu{\v{s}}kovi{\'{c}}]{chudnovsky2}
M.~Chudnovsky, G.~Cornu{\'{e}}jols, X.~Liu, P.~Seymour, and
  K.~Vu{\v{s}}kovi{\'{c}}.
\newblock Recognizing {Berge} graphs.
\newblock \emph{Combinatorica}, 25\penalty0 (2):\penalty0 143--186, 2005.

\bibitem[Chudnovsky et~al.(2006)Chudnovsky, Robertson, Seymour, and
  Thomas]{strongperfectgraphtheorem}
M.~Chudnovsky, N.~Robertson, P.~Seymour, and R.~Thomas.
\newblock The strong perfect graph theorem.
\newblock \emph{Ann. of Math. (2)}, 164\penalty0 (1):\penalty0 51--229, 2006.

\bibitem[Cozzens and Kelleher(1990)]{cozzens}
M.~B. Cozzens and L.~L. Kelleher.
\newblock Dominating cliques in graphs.
\newblock \emph{Discrete Math.}, 86\penalty0 (1-3):\penalty0 101--116, 1990.

\bibitem[Dunn et~al.(2017)Dunn, Larsen, and Nordstrom]{dunn2017}
C.~Dunn, V.~Larsen, and J.~F. Nordstrom.
\newblock Introduction to competitive graph coloring.
\newblock In A.~Wootton, V.~Peterson, and C.~Lee, editors, \emph{A Primer for
  Undergraduate Research}, pages 99--126. Springer International Publishing,
  Cham, 2017.

\bibitem[{Erd\"{o}s} et~al.(2004){Erd\"{o}s}, Faigle, {Hochst\"{a}ttler}, and
  Kern]{erdoesetal}
P.~{Erd\"{o}s}, U.~Faigle, W.~{Hochst\"{a}ttler}, and W.~Kern.
\newblock Note on the game chromatic index of trees.
\newblock \emph{Theoret. Comput. Sci.}, 313:\penalty0 371--376, 2004.

\bibitem[Faigle et~al.(1993)Faigle, Kern, Kierstead, and Trotter]{faigle1}
U.~Faigle, U.~Kern, H.~Kierstead, and W.~T. Trotter.
\newblock On the game chromatic number of some classes of graphs.
\newblock \emph{Ars Combin.}, 35:\penalty0 143--150, 1993.

\bibitem[Gardner(1981)]{gardner}
M.~Gardner.
\newblock Mathematical games.
\newblock \emph{Scientific American}, 244\penalty0 (4):\penalty0 23--26, 1981.

\bibitem[Golumbic(1978)]{golumbictrivially}
M.~C. Golumbic.
\newblock Trivially perfect graphs.
\newblock \emph{Discrete Math.}, 24\penalty0 (1):\penalty0 105--107, 1978.

\bibitem[Golumbic(2004)]{golumbic}
M.~C. Golumbic.
\newblock \emph{Algorithmic graph theory and perfect graphs}.
\newblock Elsevier, Amsterdam, 2nd edition, 2004.

\bibitem[Gr{\"{o}}tschel et~al.(1981)Gr{\"{o}}tschel, Lov{\'{a}}sz, and
  Schrijver]{gls}
M.~Gr{\"{o}}tschel, L.~Lov{\'{a}}sz, and A.~Schrijver.
\newblock The ellipsoid method and its consequences in combinatorial
  optimization.
\newblock \emph{Combinatorica}, 1\penalty0 (2):\penalty0 169--197, 1981.

\bibitem[Guan and Zhu(1999)]{guanzhu}
D.~J. Guan and X.~Zhu.
\newblock Game chromatic number of outerplanar graphs.
\newblock \emph{J. Graph Theory}, 30:\penalty0 67--70, 1999.

\bibitem[Havet and Zhu(2013)]{havetzhu}
F.~Havet and X.~Zhu.
\newblock The game grundy number of graphs.
\newblock \emph{J. Combin. Optim.}, 25:\penalty0 752--765, 2013.

\bibitem[Hochst{\"{a}}ttler and Tinhofer(1995)]{hochstaettler1995}
W.~Hochst{\"{a}}ttler and G.~Tinhofer.
\newblock Hamiltonicity in graphs with few {$P_4$}'s.
\newblock \emph{Computing}, 54\penalty0 (3):\penalty0 213--225, 1995.

\bibitem[Karp(1972)]{karp}
R.~M. Karp.
\newblock Reducibility among combinatorial problems.
\newblock In R.~E. Miller and J.~W. Thatcher, editors, \emph{Complexity of
  computer computations}, pages 85--103. Plenum Press New York, 1972.

\bibitem[Krishnamoorthy(1975)]{krishnamoorthy}
M.~S. Krishnamoorthy.
\newblock An {NP}-hard problem in bipartite graphs.
\newblock \emph{SIGACT News}, 7\penalty0 (1):\penalty0 26, 1975.

\bibitem[Lock(2016)]{lockbachelor}
E.~Lock.
\newblock \emph{The structure of {$g_B$}-perfect graphs}.
\newblock Bachelor's thesis, FernUniversit{\"{a}}t in Hagen, 06 2016.

\bibitem[Lov{\'{a}}sz(1972)]{lovasz}
L.~Lov{\'{a}}sz.
\newblock Normal hypergraphs and the perfect graph conjecture.
\newblock \emph{Discrete Math.}, 2:\penalty0 253--267, 1972.

\bibitem[Sidorowicz(2007)]{cactuses}
E.~Sidorowicz.
\newblock The game chromatic number and the game colouring number of cactuses.
\newblock \emph{Information Processing Letters}, 102\penalty0 (4):\penalty0
  147--151, 2007.

\bibitem[Sidorowicz(2010)]{sidorowiczhusimi}
E.~Sidorowicz.
\newblock Colouring game and generalized colouring game on graphs with
  cut-vertices.
\newblock \emph{Discuss. Math. Graph Theory}, 30:\penalty0 499--533, 2010.

\bibitem[Tuza and Zhu(2015)]{tuzazhu}
Z.~Tuza and X.~Zhu.
\newblock Colouring games.
\newblock In L.~W. Beineke and R.~J. Wilson, editors, \emph{Topics in chromatic
  graph theory}, chapter~14, pages 304--326. Cambridge Univ. Press, Cambridge,
  2015.

\bibitem[Wolk(1965)]{wolk}
E.~S. Wolk.
\newblock A note on ``the comparability graph of a tree''.
\newblock \emph{Proc. Amer. Math. Soc.}, 16:\penalty0 17--20, 1965.

\bibitem[Zhu(1999)]{zhuplanar}
X.~Zhu.
\newblock The game coloring number of planar graphs.
\newblock \emph{J. Combin. Theory, Ser. B}, 75:\penalty0 245--258, 1999.

\bibitem[Zhu(2000)]{zhupseudopartial}
X.~Zhu.
\newblock The game coloring number of pseudo partial {$k$}-trees.
\newblock \emph{Discrete Math.}, 215:\penalty0 245--262, 2000.

\bibitem[Zhu(2008{\natexlab{a}})]{zhucartesian}
X.~Zhu.
\newblock Game coloring the cartesian product of graphs.
\newblock \emph{Journal of Graph Theory}, 59:\penalty0 261--278,
  2008{\natexlab{a}}.

\bibitem[Zhu(2008{\natexlab{b}})]{zhurefined}
X.~Zhu.
\newblock Refined activation strategy for the marking game.
\newblock \emph{J. Combin. Theory, Ser. B}, 98\penalty0 (1):\penalty0 1--18,
  2008{\natexlab{b}}.

\end{thebibliography}
\label{sec:biblio}

\end{document}